\documentclass[11pt, reqno]{amsart}
\usepackage[dvipsnames]{xcolor}
\usepackage{geometry, amsmath, amsfonts, amssymb, tikz-cd, mathtools,lscape,multirow}
\usepackage{mathrsfs}
\usepackage{graphicx, rotating}
\usepackage{cancel}
\usepackage[bordercolor=gray!20,backgroundcolor=blue!10,linecolor=blue!50,textsize=footnotesize,textwidth=0.8in]{todonotes}
\usepackage{hyperref} 
\usepackage{caption}
\usepackage{subcaption}
\usepackage{bm}
\usepackage{comment}
\usepackage{microtype}
\usepackage{colonequals}

\usepackage[numbers]{natbib} 
\usepackage{url} 

\newtheorem{thm}{Theorem}[section]
\newtheorem{theorem}[thm]{Theorem}

\newtheorem{cor}[thm]{Corollary}

\newtheorem{lem}[thm]{Lemma}
\newtheorem{lemma}[thm]{Lemma}
\newtheorem{prop}[thm]{Proposition}

\theoremstyle{definition}

\newtheorem{definition}[thm]{Definition}
\newtheorem{ex}[thm]{Example}

\newtheorem{rem}[thm]{Remark}

\newtheorem{remark}[thm]{Remark}

\numberwithin{equation}{section}

\newcommand{\rs}{{\stackrel{\rm r}{\rightharpoonup}}} 
\newcommand{\ls}{{\stackrel{\rm l}{\rightharpoonup}}} 
\newcommand{\lrs}{{\stackrel{\rm lr}{\rightharpoonup}}} 

\newcommand{\rh}{\rightharpoonup}

\newcommand{\marginal}[1]{\marginpar{\tiny #1}}

\begin{document}

\title[Agol cycles of pseudo-Anosov maps]
{Agol cycles of pseudo-Anosov maps on the 2-punctured torus and 5-punctured sphere}

\author{Jean-Baptiste Bellynck}
\address {Department of Mathematics, Ludwig-Maximilians-University Munich, 80333 Munich, Theresienstr. 39, Germany}
\email{j.bellynck@campus.lmu.de}

\author{Eiko Kin} 
\address {Center for Education in Liberal Arts and Sciences, Osaka University, Toyonaka, Osaka 560-0043, Japan}
\email{kin.eiko.celas@osaka-u.ac.jp}


\date{\today}

\begin{abstract}
Given a periodic splitting sequence of a measured train track, an Agol cycle is the part that constitutes a period up to the action of a pseudo-Anosov map and the rescaling by its dilatation. We consider a family of pseudo-Anosov maps on the $2$-punctured torus and on the $5$-punctured sphere. 
We present measured train tracks and compute their Agol cycles. We give a condition under which two maps in the defined family are conjugate or not. In the process, we find a new formula for the dilatation.   
\end{abstract}

\maketitle

\section{Introduction}
\label{section_Introduction}
Let $\Sigma = \Sigma_{g, n}$ be an orientable surface with genus $g $ and $n $ punctures. 
Let $\mathrm{MCG}(\Sigma)$ be the mapping class group of $\Sigma$. 
According to the Nielsen-Thurston-classification, 
every element of $\mathrm{MCG}(\Sigma)$ falls into $3$ types: periodic, reducible and pseudo-Anosov. 
If $\phi \colon \Sigma \to \Sigma$ is a pseudo-Anosov map, then there exist associated stable and unstable measured laminations 
$(\mathcal{L}^s, \nu^s)$ and $(\mathcal{L}^u, \nu^u)$ and the dilatation $\lambda = \lambda(\phi) >1$ such that 
$$\phi(\mathcal{L}^s, \nu^s)= (\mathcal{L}^s, \lambda \nu^s) \hspace{2mm} \mbox{and} \hspace{2mm} \phi(\mathcal{L}^u, \nu^u)= (\mathcal{L}^u, \lambda^{-1} \nu^u).$$

A {\em measured train track} $(\tau, \mu)$ is a train track $\tau$ with a  transverse measure $\mu$. 
Edges of a train track are called {\em branches} and vertices are called {\em switches}. 
A branch that locally looks like the central branch in Figure \ref{fig_split_shift}(1) is called a {\em large branch}.  
A {\em splitting} at a large branch 
is an operation that gives a new measured train track. 
There are two kinds of splitting, {\em left} and {\em right splitting} at a large branch (Figure~\ref{fig_split_shift}(2)(3)). 
See Definition~\ref{def_operation}.

A {\em maximal splitting} $(\tau_0, \mu_0) \rightharpoonup (\tau_1, \mu_1)$ is an operation on the measured train track $(\tau_0, \mu_0)$ 
that splits all the large branches that carry  maximal $\mu_0$-weight 
and $(\tau_1, \mu_1)$ is the resulting measured train track.
If all the splittings in a maximal splitting are left (resp. right) splittings, the maximal splitting is denoted by $\ls$ (resp. $\rs$) 
and called a {\em left (resp. right) maximal splitting}.

\begin{figure}[ht]
\centering
\includegraphics[height=2cm]{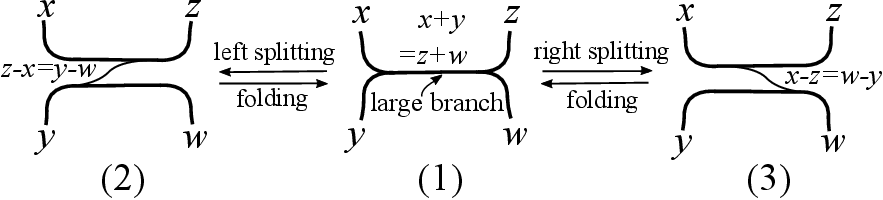}
\caption{
(1) A large branch. 
(2) Left splitting when $z>x$ ($\Leftrightarrow y>w$), (3) right splitting  when $x>z$ ($\Leftrightarrow w>y$) at the large branch. 
} 
\label{fig_split_shift}
\end{figure}

It was proven by Agol that after enough maximal splittings 
the measured train track $(\tau, \mu)$ suited to the stable measured lamination of a pseudo-Anosov map $\phi$ 
will have changed to $\phi(\tau, \lambda^{-1}\mu)\colonequals   (\phi(\tau), \lambda^{-1}\phi_*(\mu))$, 
where the measure $\phi_*(\mu)$ is defined by $\phi_*(\mu)(e)\colonequals  \mu(\phi^{-1}(e))$ for a branch $e$ in the train track $\phi(\tau)$. 
To state Agol's result precisely, 
a sequence of consecutive $n$ maximal splittings
$(\tau_0, \mu_0) \rightharpoonup  \cdots \rightharpoonup (\tau_n, \mu_n) $ 
is denoted by $(\tau_0,\mu_0) \rightharpoonup^n (\tau_n, \mu_n)$.

\begin{theorem}[Agol \cite{Agol11}. See also Agol-Tsang \cite{AgolTsang22}]
\label{thm:Agol}
Let $\phi \colon \Sigma \to \Sigma$ be a pseudo-Anosov map with dilatation $\lambda$. 
Let $(\tau, \mu)$ be a measured train track 
suited to the stable measured lamination of $\phi$. 
Then   there exist  $n\geq 0$ and $m >0$ such that 
$$(\tau, \mu)  \rightharpoonup^n (\tau_n, \mu_n) \rightharpoonup^m (\tau_{n+m}, \mu_{n+m})= \phi(\tau_n, \lambda^{-1}\mu_n).$$ 
\end{theorem}

For the terminology {\em suited to}, see Definition~\ref{definition:suited}. 
We call the maximal splitting sequence 
$$ (\tau_n, \mu_n) \rightharpoonup^m (\tau_{n+m}, \mu_{n+m})  \rightharpoonup^m (\tau_{n+2m}, \mu_{n+2m}) \rightharpoonup^m \cdots$$
a {\em periodic splitting sequence} of $\phi$.  We  call the finite subsequence 
$(\tau_n, \mu_n) \rightharpoonup^m  (\tau_{n+m}, \mu_{n+m})$ 
an {\em Agol cycle} of $\phi$ and call $m$ the {\em Agol cycle length} of $\phi$, denoted by $\ell(\phi)$. 
The {\em total splitting number} of an Agol cycle  of $\phi$, 
denoted by $N(\phi)$, 
is the number of large branches that are split in the Agol cycle (Definition~\ref{definition_splitting-number}(3)).

An {\em equivalence class} of an Agol cycle  is a conjugacy invariant of pseudo-Anosov maps (Section~\ref{subsection_traintrack}). 
The Agol cycle length $\ell(\phi)$ and total splitting number $N(\phi)$ are conjugacy invariants as well.  
If  $\phi: \Sigma \rightarrow \Sigma$ is  fully-punctured 
(i.e., the singularities of the stable/unstable foliations of $\phi$ lie on the punctures of $\Sigma$), 
$N(\phi)$ equals the number of ideal tetrahedra 
in the  veering triangulation of the mapping torus of $\phi$. 
See \cite{Agol11} for more details.

It is natural to ask how the Agol cycle length $\ell(\phi)$ and total splitting number $N(\phi)$  relate to other invariants of pseudo-Anosov maps. 
In  \cite{KawamuroKin23} it was proven that 
for every pseudo-Anosov $3$-braid $\beta$, 
its Agol cycle length, the Garside canonical length of any element in the super summit set  of $\beta$ are the same. 
Agol-Tsang \cite{AgolTsang22} proved that the total splitting number $N(\phi)$ for a fully punctured pseudo-Anosov $\phi: \Sigma \rightarrow \Sigma$ 
 is bounded from above by a constant depending on the normalized dilatation $\lambda^{-\chi (\Sigma)}$, 
where $\chi(\Sigma)$ is the Euler characteristic of  $\Sigma$.

The main goal of this paper is to give an explicit description of an Agol cycle of every pseudo-Anosov map in the two semigroups 
$F_T \subset \mathrm{MCG}(\Sigma_{1,2})$ and $F_D \subset \mathrm{MCG}(\Sigma_{0,5})$ which will be defined below. 
On the $2$-punctured torus $\Sigma_{1, 2}$, let $\delta_i$ be the right-handed Dehn twist about the simple closed curve $c_i \subset \Sigma_{1, 2}$ 
for $i \in \{1,2,3\}$ shown in Figure~\ref{fig_t_track}(1). 
The hyperelliptic involution exchanges the two punctures of the torus and induces a $2$-fold branched cover $\Sigma_{1, 2} \to \Sigma_{0, 5}$ of the $5$-punctured sphere. 
Then  $\delta_i$ descends to $\sigma_i$, the positive half-twist about the segment $\alpha_{i}$ connecting the punctures $i$ and $i+1$ (Figure~\ref{fig_t_track}(5)). 

 \begin{figure}[htbp]
    \centering
    \includegraphics[height=3.5cm]{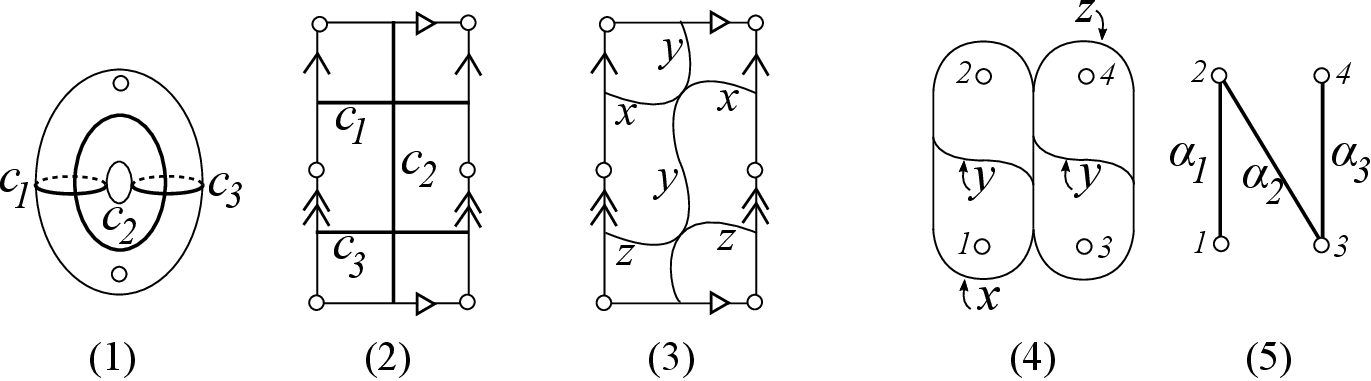}
    \caption{(1)(2) Simple closed curves $c_1, c_2$ and $c_3$ in $\Sigma_{1,2}$. 
    (3) $(\mathfrak{b}, \bm{x})$ in $\Sigma_ {1,2}$ and 
    (4) $(\mathfrak{b}_L, \bm{x})$  in $\Sigma_{0,5}$ 
    for  $\bm{x}= \left(\begin{smallmatrix} x \\ y \\ z \end{smallmatrix}\right)$. 
    (5)  Segments $\alpha_i$ in $\Sigma_ {0,5}$.}
    \label{fig_t_track}
    \end{figure}
    We study pseudo-Anosov maps 
    in the semigroups
    $$F_T \colonequals  F(\delta_1, \delta_3, \delta_2^{-1}) \subset \text{MCG}(\Sigma_{1, 2}) \text{ and } F_D \colonequals  F(\sigma_1, \sigma_3, \sigma_2^{-1})\subset \text{MCG}(\Sigma_{0, 5})$$
    generated by $\delta_1$, $\delta_3$ and $\delta_2^{-1}$  and 
    by $\sigma_1$, $\sigma_3$ and $\sigma_2^{-1}$. 
 Each $\sigma_i$ for $i \in \{1,2,3\}$ fixes the fifth puncture of $\Sigma_{0,5}$. 
 Hence, one can regard an element of $F_D$ as a mapping class on the $4$-punctured disk.  
    The subset $\mathcal{I}_n \subset \mathbb{N}_0^{3n}$, 
    where $ \mathbb{N}_0 = {\mathbb N} \cup \{0\}$, will be useful for our study of pseudo-Anosov maps 
    in $F_T$ and $F_D$ (Definition~\ref{definition_In}). 
    For each $\boldsymbol{p} = (p_n, p_n', q_n, ..., p_1, p_1', q_1)\in \mathcal{I}_n$ 
    $$\Phi_{\boldsymbol{p}} \colonequals  \delta_1^{p_n}\delta_3^{p_n'}\delta_2^{-q_n} \cdots \delta_1^{p_1}\delta_3^{p_1'}\delta_2^{-q_1} \in F_T \text{ and } 
    \phi_{\boldsymbol{p}} \colonequals   \sigma_1^{p_n}\sigma_3^{p_n'}\sigma_2^{-q_n} \cdots \sigma_1^{p_1}\sigma_3^{p_1'}\sigma_2^{-q_1} \in F_D$$
    are pseudo-Anosov maps. 
We take  matrices 
$M_1 = \left(\begin{smallmatrix}1 & 1 & 0 \\ 0 & 1 & 0 \\ 0 & 0 & 1\end{smallmatrix}\right)$, $M_3 = \left(\begin{smallmatrix}1 & 0 & 0 \\ 0 & 1 & 0 \\ 0 & 1 & 1\end{smallmatrix}\right)$, $M_2 = \left(\begin{smallmatrix}1 & 0 & 0 \\ 1 & 1 & 1 \\ 0 & 0 & 1\end{smallmatrix}\right)$. 
For each $\boldsymbol{p} \in \mathcal{I}_n$ 
the matrix 
$$M_{\boldsymbol{p}} \colonequals  M_1^{p_n}M_3^{p_n'}M_2^{q_n} \cdots M_1^{p_1}M_3^{p_1'}M_2^{q_1}$$ 
is Perron-Frobenius. 
    The Perron-Frobenius eigenvalue $\lambda_{\boldsymbol{p}}$ is equal to 
    the dilatations of  maps $\Phi_{\boldsymbol{p}}$ and $\phi_{\boldsymbol{p}}$. 
In Theorem~\ref{thm_expansion} we present an explicit description of the Perron-Frobenius eigenvalue 
    $\lambda_{\boldsymbol{p}}$ and the normalized eigenvector $\bm{v}_{\bm{p}}$. 
    As a consequence 
    we see that 
    $\lambda_{\bm{p}}$ is a quadratic irrational  (Remark~\ref{rem_quadratic-irrational}).     
 Let $\mathfrak{b} \subset \Sigma_{1, 2}$ (resp. $(\mathfrak{b}_L \in \Sigma_{0,5}$) 
 be train track as in Figure \ref{fig_t_track}(3) (resp. Figure \ref{fig_t_track}(4)). 
 Assigning the coefficients of a Perron-Frobenius  eigenvector $\boldsymbol{x}$ of $M_{\boldsymbol{p}} $  
    to the branches of the train track makes the measured train track $(\mathfrak{b}, \boldsymbol{x})$ 
    (resp. $(\mathfrak{b}_L, \boldsymbol{x})$). 

We say that $\bm{p} \in  \mathcal{I}_n$ is {\em symmetric} if $p_i= p_i' $ for all $i \in \{1, \dots, n\}$. 
    Otherwise,  $\bm{p}$ is {\em asymmetric}.  To state our results, 
we use the symbol $\ls^n$ (resp. $\rs^n$) for $n$ consecutive left (resp. right) maximal splittings.

\begin{theorem}
\label{thm_2-torus}
For $ \bm{p}=  (p_n, p_n', q_n, \dots, p_1, p_1', q_1)  \in  \mathcal{I}_n$ 
let 
$\Phi_{\bm{p}} \in F_T$ be the pseudo-Anosov map and 
$M_{\bm{p}}$ be the Perron-Frobenius matrix associated with $ \bm{p}$. 
Let ${\bm v} >\bm{0}$ be an eigenvector with respect to the Perron-Frobenius eigenvalue $\lambda_{\bm{p}}$ of  $M_{\bm{p}}$. 
Then  the Agol cycle length $\ell$ of $\Phi_{\bm{p}}$ is 
$$  \ell= 
\left\{
\begin{array}{ll}
 \sum_{i=1}^n (p_i + 2  q_i) 
& \mbox{if}\ \bm{p}\ \mbox{is symmetric}, 
\\
 \sum_{i=1}^n (p_i+ p_i' + 3  q_i)
& \mbox{if}\ \bm{p}\ \mbox{is asymmetric}. 
\end{array}
\right.
$$
Moreover, starting with the measured train track $(\tau_0, \mu_0) = (\mathfrak{b}, \lambda_{\bm{p}}\bm{v})$, 
a finite subsequence of the maximal splitting sequence 
\begin{eqnarray*}
 (\tau_0, \mu_0) \rs^{p_n} \ls^{2q_n}  \cdots \rs^{p_1} \ls^{2q_1} (\tau_{\ell}, \mu_{\ell}) \hspace{1cm}
& \mbox{if}& \hspace{-2mm} \bm{p}\ \mbox{is symmetric}, 
\\
(\tau_0, \mu_0) \rs^{p_n+p_n'} \ls^{3q_n}  \cdots \rs^{p_1+p_1'} \ls^{3q_1} (\tau_{\ell}, \mu_{\ell})
& \mbox{if}& \hspace{-2mm} \bm{p}\ \mbox{is asymmetric}
\end{eqnarray*}
forms an Agol cycle of $\Phi_{\bm{p}}$. 
\end{theorem}

We later prove an analogous statement for the pseudo-Anosov maps $\phi_{\bm{p}} \in F_D$ inside the semi-group $F_D$ (See Theorem~\ref{thm_5-punctured-sphere-precise}).

As applications, 
we give  formulas on the total  splitting numbers $N(\Phi_{\bm{p}})$ and $N(\phi_{\bm{p}})$ for each $\bm{p} \in \mathcal{I}_n$ 
(Theorems~\ref{thm_total_FT}, \ref{thm_total_FD}). 
We also classify conjugacy classes of pseudo-Anosov maps in  $F_T$ and $F_D$ (Theorem~\ref{thm_conjugacy-class}). 
The total splitting numbers $N(\Phi_{\bm{p}})$ and $N(\phi_{\bm{p}})$ have the following additive property. 

\begin{thm}
\label{thm_additive}
For $\bm{p}=  (p_n, p_n', q_n, \dots, p_1, p_1', q_1)  \in  \mathcal{I}_n$ and  
$\bm{t}= (t_m, t_m', u_m, \dots, t_1, t_1', u_1)  \in  \mathcal{I}_m$, 
we set 
$\bm{pt} \colonequals  (p_n, p_n', q_n, \dots, p_1, p_1', q_1, t_m, t_m', u_m, \dots, t_1, t_1', u_1) \in \mathcal{I}_{n+m}$. 
The total splitting number of $\Phi_{\bm{pt}} \in F_T$  
satisfies 
$N(\Phi_{\bm{p t}}) = N(\Phi_{\bm{p }})+ N(\Phi_{\bm{ t}})$. 
A parallel statement holds for  $\phi_{\bm{pt}} \in F_D$.  
\end{thm}

The paper is organized as follows. 
In Section~\ref{section_preliminaries} 
we recall basic definitions and prove  lemmas. 
In Sections~\ref{section_FT} and \ref{section_FD} 
we compute Agol cycles of pseudo-Anosov maps in $F_T$ and  $F_D$. 
In Section~\ref{section_applications} 
we classify pseudo-Anosov conjugacy classes in  $F_T$ and $F_D$.

\section{Preliminaries}
\label{section_preliminaries}

The mapping class group $\mathrm{MCG}(\Sigma)$ of a surface $\Sigma= \Sigma_{g,n}$ 
is the group of isotopy classes of orientation preserving homeomorphisms of $\Sigma$ preserving the punctures setwise. 
We apply elements of the mapping class group from right to left; i.e., the product  $fg$ means that we apply $g$, then $f$. 
For simplicity we do not distinguish between a homeomorphism $\phi: \Sigma \rightarrow \Sigma$ 
and its  mapping class $[\phi] \in \mathrm{MCG}(\Sigma)$.

\subsection{Measured train tracks}
\label{subsection_traintrack}

A {\em train track} $\tau \subset \Sigma$ is a finite $C^1$-embedded graph, equipped with a well-defined tangent line at each vertex, 
also satisfying some additional properties as stated in Penner-Harer \cite{PennerHarer92}.  
In this paper we assume our train tracks to be  trivalent. A {\em measured train track} $(\tau, \mu)$ is a train track $\tau$ with a measure $\mu$. This is a function that assigns a positive weight to each branch. 
Measured train tracks are required to satisfy the {\em switch condition}. 
This means that if two branches $a, b$ merge into one branch $c$, then the weights satisfy $\mu(a) + \mu(b) = \mu(c)$. 
See Figure~\ref{fig_switch_shift}(1).

\begin{figure}[ht]
\centering
\includegraphics[height=1.8cm]{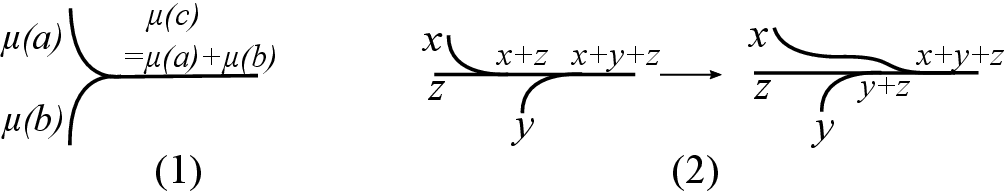}
\caption{(1) Switch condition. 
(2) Shifting.} 
\label{fig_switch_shift}
\end{figure}

\begin{definition}
\label{def_operation}
We consider a large branch as in Figure \ref{fig_split_shift}(1).  Depending on weights $x, y, z$ and $w$ in Figure \ref{fig_split_shift}(1), 
a {\em splitting} divides a large branch into two branches and connects the two parts with either a left-facing or right-facing branch, 
thereby preserving the switch condition. 
Depending on the type of a branch inserted, the splitting is called a {\em left} or {\em right splitting} at a large branch (Figure~\ref{fig_split_shift}(2)(3)).  
Similarly, we can produce new measured train tracks through the use of {\em folding} (Figure~\ref{fig_split_shift})  
and {\em shifting} (Figure \ref{fig_switch_shift}(2)).  
\end{definition}

Recall that 
if all the splittings in a maximal splitting $(\tau_0, \mu_0) \rightharpoonup (\tau_1, \mu_1)$ are left (resp. right) splittings, the maximal splitting is denoted by $\ls$ (resp. $\rs$) 
and called a left (resp. right) maximal splitting.
If  there exist both left and right splittings, the maximal splitting is denoted by $\lrs$ 
and called  a {\em mixed maximal splitting}.

Measured train tracks 
$(\tau, \mu)$, $(\tau', \mu')$ in $\Sigma$ are {\em equal} (and write $(\tau, \mu) = (\tau', \mu')$) 
if there exists a diffeomorphism $f: \Sigma \rightarrow \Sigma$ isotopic to the identity map such that 
$f(\tau, \mu)= (\tau', \mu')$.

Measured train tracks $(\tau, \mu)$, $(\tau', \mu')$ in $\Sigma$ are {\em equivalent} if they are related to each other 
by a sequence of splittings, foldings, shiftings and isotopies.  
Thus measured train tracks in a splitting sequence are equivalent. 
Equivalence classes of measured train tracks are in one-to-one correspondence with measured laminations \cite[Theorem~2.8.5]{PennerHarer92}. 

\begin{definition}
\label{definition:suited}
Let $(\mathcal{L}, \nu)$ be a measured lamination in $\Sigma$, 
and let $(\tau, \mu)$ be a measured train track in $\Sigma$. 
Then  
$(\tau, \mu)$ is {\em suited} to $(\mathcal{L},\nu)$ if there exists a differentiable map 
$f: \Sigma \rightarrow \Sigma$ 
homotopic to the identity map on $\Sigma$ with the following conditions: 
\begin{itemize}
\item 
$f(\mathcal{L}) = \tau$. 

\item 
$f$ is nonsingular on the tangent spaces to the leaves of $\mathcal{L}$. 

\item 
If $p$ is an interior point of a branch $e$ of $\tau$ 
then $\nu(f^{-1}(p)) = \mu(e)$. 
\end{itemize}
\end{definition}

\begin{definition}
\label{definition_splitting-number} \hfill
\begin{enumerate}
\item
The {\em splitting number of  a maximal splitting} $(\tau_0, \mu_0) \rightharpoonup (\tau_1, \mu_1)$  is the number of large branches split, i.e., the number of the large branches of $(\tau_0, \mu_0)$ with maximal weight. 

\item 
The {\em total splitting number of a finite sequence of maximal splittings} $(\tau,\mu) \rightharpoonup^n (\tau_n, \mu_n)$ 
is the sum of the splitting numbers over all maximal splittings in the finite sequence.

\item 
The {\em total splitting number of an Agol cycle} $(\tau_n, \mu_n) \rightharpoonup^m  (\tau_{n+m}, \mu_{n+m})$ of $\phi$, 
denoted by $N(\phi)$, is the sum of the splitting numbers over all maximal splittings $(\tau_{n+i}, \mu_{n+i}) \rightharpoonup (\tau_{n+i+1}, \mu_{n+i+1})$ in the Agol cycle. 
The Agol cycle length $\ell(\phi)$ is less than or equal to $N(\phi)$. The equality holds if and only if the splitting number of each maximal splitting in the Agol cycle is exactly $1$. 
\end{enumerate}
\end{definition}

\begin{definition}
\label{definition_ combinatorially-isomorphic}
Let $\phi, \phi'\colon \Sigma \to \Sigma$ be pseudo-Anosov maps with 
periodic splitting sequences 
$$ \mathscr{P}: (\tau_n, \mu_n) \rh^m (\tau_{n+m}, \mu_{n+m})= \phi(\tau_n, \lambda^{-1}\mu_n)\  \rightharpoonup  \cdots$$
 of  $\phi$ and 
$$\mathscr{P}': (\tau'_{n'}, \mu'_{n'}) \rh^{m'} (\tau'_{n'+m'}, \mu'_{n'+m'})= \phi'(\tau'_{n'},  (\lambda')^{-1}\mu'_{n'})\ \rightharpoonup  \cdots$$ 
of $\phi'$. 
We say that $\mathscr{P}$ and $\mathscr{P}'$ are {\em combinatorially isomorphic} (\cite{HodgsonIssaSegerman16})  
if $m=m'$ is fulfilled and 
there exist an orientation-preserving diffeomorphism 
$h\colon\Sigma \to \Sigma$, integers $p, q \in {\mathbb Z}_{\geq 0}$ and $c\in {\mathbb R}_{>0}$ such that the following conditions (1) and (2) hold.
\begin{enumerate}
\item[(1)] 
$\phi'=h \circ \phi \circ h^{-1}$. 

\item[(2)] 
$h(\tau_{i+p}, \mu_{i+p})=(\tau'_{i+q}, c\mu'_{i+q})$ for all $i\in {\mathbb Z}_{\geq 0}$.
\end{enumerate} 
We say that two Agol cycles  
$(\tau_n, \mu_n) \rh^m (\tau_{n+m}, \mu_{n+m})$ of $\phi$ and $(\tau'_{n'}, \mu'_{n'}) \rh^{m'} (\tau'_{n'+m'}, \mu'_{n'+m'})$ of $\phi'$ are {\em equivalent} 
if $m=m'$ is fulfilled and there exist an orientation-preserving diffeomorphism $h\colon\Sigma \to \Sigma$,  integers $p, p' \in {\mathbb Z}_{\geq 0}$ and $c\in {\mathbb R}_{>0}$ such that $h(\tau_{n+p}, \mu_{n+p})=(\tau'_{n'+ p'}, c\mu'_{n'+ p'})$. 
The condition for equivalent Agol cycles implies condition (2). See \cite[Lemma~2.2]{KawamuroKin23}.
\end{definition}

\begin{thm}[Theorem~5.3 in Hodgson-Issa-Segerman \cite{HodgsonIssaSegerman16}]
\label{thm_Hodgson-Issa-Segerman}
 Pseudo-Anosov maps 
$\phi, \phi'\colon\Sigma \to \Sigma$ are conjugate in $\mathrm{MCG}(\Sigma)$ 
if and only if $\mathscr{P}$ and $\mathscr{P}'$  are combinatorially isomorphic. 
\end{thm}
As a consequence,  the equivalence class of an Agol cycle of $\phi$ is a conjugacy invariant. 
The  Agol cycle length $\ell(\phi)$ and total splitting number $N(\phi)$ are conjugacy invariants as well, 
since they are equal for equivalent Agol cycles. 

When we regard a maximal splitting $(\tau, \mu)\rightharpoonup (\tau', \mu')$ 
as an operation on the measured train track, 
we write
$ (\tau', \mu')=\  \rightharpoonup (\tau, \mu)$.  
We write $n$ consecutive left (resp. right) maximal splittings 
$(\tau, \mu) \ls^n (\tau_n, \mu_n)$ (resp. $(\tau, \mu) \rs^n (\tau_n, \mu_n)$) as 
$(\tau_n, \mu_n) = \  \ls^n (\tau, \mu)$ (resp. $(\tau_n, \mu_n) =\   \rs^n (\tau, \mu)$).We also write a finite sequence $(\tau, \mu) \ls^n (\tau_n, \mu_n) \rs^m (\tau_{n+m}, \mu_{n+m})$ as 
$ (\tau_{n+m}, \mu_{n+m}) = \ \rs^m  \circ  \ls^n (\tau, \mu)$.

The operation $\rightharpoonup$ and a diffeomorphism $\phi: \Sigma \rightarrow \Sigma$ commute 
on measured train tracks:

\begin{lemma}[Lemma 2.1 in \cite{KawamuroKin23}]
\label{lem_commute}
Let $(\tau, \mu)$ be a measured train track in $\Sigma$ and 
$\phi: \Sigma \rightarrow \Sigma$  an orientation-preserving diffeomorphism. If $(\tau, \mu)$ admits consecutive $n$ left maximal splittings, then we have 
$(\phi  \circ {\ls^n}) (\tau, \mu) = ({\ls^n} \circ \phi)  (\tau ,\mu)$. A parallel statement holds for $\rs^n$. 
\end{lemma}

\begin{remark}
\label{rem_commute} (This remark is used for the proof of Theorem~\ref{thm_conjugacy-class}.)  
By Lemma~\ref{lem_commute} we have the following commutative diagram: 
$$
\begin{array}{ccccccc}
(\tau, \mu) & \rightharpoonup &  (\tau_1, \mu_1) &  \rightharpoonup & \cdots & \rightharpoonup &(\tau_n, \mu_n) \\ 
 \downarrow     &                                & \downarrow &                               &            &                           &   \downarrow         \\\phi(\tau, \mu) & \rightharpoonup &  \phi(\tau_1, \mu_1)&  \rightharpoonup & \cdots & \rightharpoonup & \phi(\tau_n, \mu_n)
\end{array} 
$$
Lemma~\ref{lem_commute}  tells us that 
the (left, right, mixed) type of the maximal splitting 
$\phi(\tau_i, \mu_i) \rightharpoonup \phi(\tau_{i+1}, \mu_{i+1}) $ 
is the same as that of $(\tau_i, \mu_i) \rightharpoonup (\tau_{i+1}, \mu_{i+1}) $. 

\end{remark}

\subsection{Perron-Frobenius matrices}
\label{subsection_PF}

we write $A \ge B$ if $a_ {rs} \ge  b_{rs}$ for all $r, s$. 
Suppose that $M$ is an $n$ by $n$ square matrix with nonnegative integer entries. 
We say that  $M$ is {\em Perron-Frobenius} if some power of $M$ is a positive matrix. 
Perron-Frobenius matrices have the following properties.

\begin{thm}[Perron-Frobenius]
\label{thm_Perron-Frobenius}
A Perron-Frobenius $M$ has a real eigenvalue $\lambda>1$
which exceeds the moduli of all other eigenvalues. 
There exists  a strictly positive eigenvector $\bm{v}$ associated with $\lambda$. 
Moreover, $\bm{v}$ is the unique  positive eigenvector of $M$ (up to positive multiples), and 
$\lambda$ is a simple root of the characteristic equation of $M$. 
\end{thm}

For the proof, see \cite{Gantmacher59}. We call $\lambda= \lambda(M)>1$  the {\em Perron-Frobenius eigenvalue} of $M$ and 
call  $\bm{v}$ a {\em Perron-Frobenius eigenvector}. 
 
 \begin{definition}
 \label{definition_In}
 For each $n \in {\mathbb N}$ 
 the subset $\mathcal{I}_n \subset \mathbb{N}_0^{3n}$ 
 is defined as follows. 
  \[ \mathcal{I}_n \colonequals \left\{ \bm{p}= (p_n, p_n', q_n, \dots, p_1, p_1', q_1)  \in {\mathbb N}_0^{3n} \middle|
 \begin{array}{l}
 \ ^{\exists} j, \ ^{\exists}k \in  \{1, \dots, n\}\ \mbox{such\ that}\ 
 p_j, p_k' >0
 \\
  p_i+p_i',  q_i >0 
 \ \mbox{for each}\ i \in \{1, \dots, n\}
  \end{array}
 \right\}\]
 For example, 
 $(1,0,2,0,1,1) \in \mathcal{I}_2$, $(1,0,2,1,0,1) \not\in \mathcal{I}_2$. 
By definition, 
 $ \mathcal{I}_1 = {\mathbb N}^{3}$. 
 \end{definition}

 We recall the matrix  
 $M_{\bm{p}} = M_1^{p_n}M_3^{p_n'}M_2^{q_n} \cdots M_1^{p_1}M_3^{p_1'}M_2^{q_1}$ for $\bm{p}  \in  \mathcal{I}_n $.

 \begin{lem}
 \label{lem_Perron-Frobenius-p}
 For each $\bm{p}  \in  \mathcal{I}_n $, $M_{\bm{p}}$ is Perron-Frobenius. 
 \end{lem}
 
 \begin{proof}
A  computation shows that 
$M_i^n \ge M_i \ge  \left(\begin{smallmatrix}1 & 0 & 0 \\0 & 1 & 0 \\0 & 0 & 1\end{smallmatrix}\right)$ 
for $n \in {\mathbb N}$ and $i = 1,2,3$. 
By definition of $\mathcal{I}_n$, 
 all the matrices $M_1, M_2$ and $M_3$ appear in the product $M_{\bm{p}}$  at least once. 
 We can check that 
 $M_{\bm{p}} \ge M_1 M_3 M_2  = M_3 M_1 M_2> \bm{0}$. 
 This means that $M_{\bm{p}}$ is positive. 
 In particular, $M_{\bm{p}}$ is Perron-Frobenius. 
 (This fact also follows from  \cite[Theorem~3.1]{Penner88}.)
 \end{proof}

In this section 
we give an explicit description of a Perron-Frobenius eigenvector of $M_{\bm{p}}$ and its eigenvalue $\lambda_{\bm{p}}$. 
To do this, we first consider the 
infinite continued fraction expansion of an irrational number $a$. 
$$
a= a_0+ \cfrac{1}{a_1 + 
           \cfrac{1}{\ddots+ 
           \cfrac{1}{a_{k}+ \cdots }}}
= [a_0, a_1, \cdots, a_k, \cdots]
$$
with $a_i \in {\mathbb Z}$ and $a_i >0$ for $i \ge 1$. 
By Lagrange's theorem, 
$a$ is a quadratic irrational if and only if 
the expansion is eventually periodic; i.e., there exists $t \ge 1$ with $a_i = a_{i+t}$ for all $i \gg 1$. 
We write a quadratic irrational 
$a= [a_0, \cdots, a_{k-1}, b_0, \cdots, b_{t-1}, b_0, \cdots, b_{t-1}, \cdots\ ]$ 
as $[a_0, \cdots, a_{k-1}, \overline{b_0, \cdots, b_{t-1}}]$. 

 Given $\bm{p} \in \mathcal{I}_n$, we next define the width $w_{\bm{p},j}$ and height $h_{\bm{p},j} $ 
 for each $j \in \mathbb{N}_0$ as follows. 
\begin{eqnarray*}
& &\mbox{For}\ j=0, \ w_{\bm{p}, 0} = 1 \ \text{ and }\ h_{\bm{p}, 0} = [0, \overline{p_n+p_n', q_n, ..., p_1 + p_1', q_1}].
\\
& &\mbox{For}\ j >0, \ w_{\bm{p}, j} = w_{\bm{p}, j-1}-(p_{n-j+1}+p_{n-j+1}')h_{\bm{p}, j-1} \ \text{ and }\ h_{\bm{p}, j} = h_{\bm{p}, j-1}-q_{n-j+1}w_{\bm{p},j}.
\end{eqnarray*}
The \textit{split ratio} $s_{\bm{p}}$ ($ 0 < s_{\bm{p}}< 1$) is defined by 
$$s_{\bm{p}} = \sum_{i = 0}^\infty p_{-i}h_{\bm{p},i},$$
where the index of $p_{-i}$ is understood to be mod $n$.

\begin{figure}[ht]
\centering
\includegraphics[height=5cm]{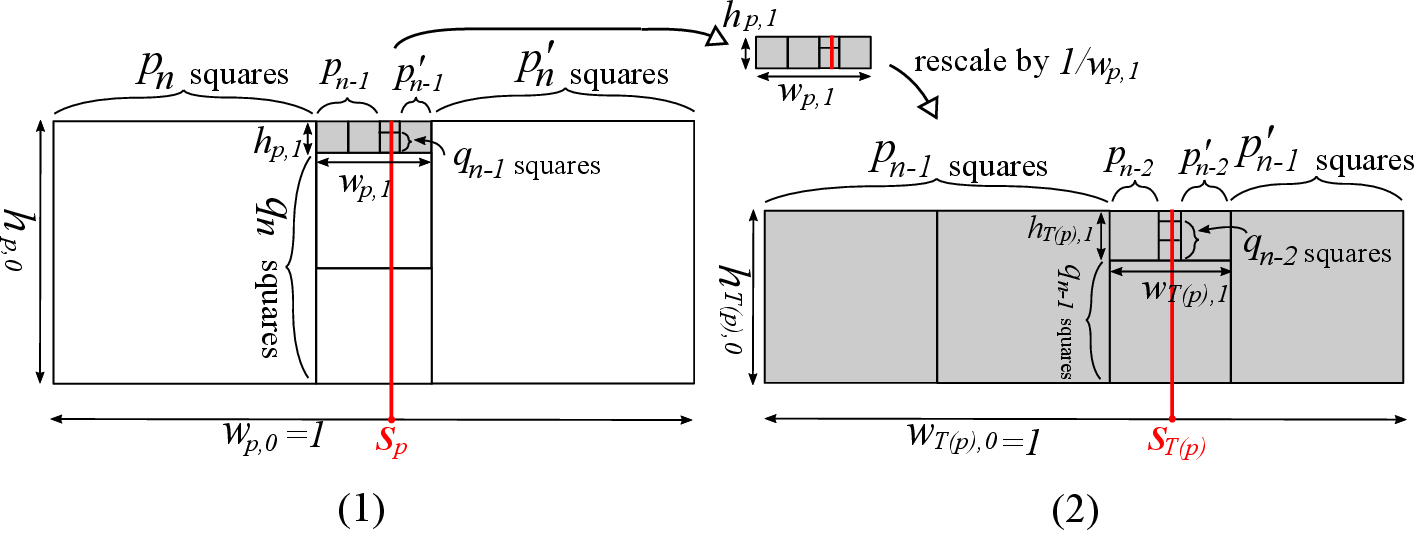}
\caption{Partitioned rectangles (1) $\mathrm{rect}(\bm{p})$, (2) $\mathrm{rect}(T(\bm{p}))$ 
for $\bm{p} = (1, 1, 2, 2, 1, 1), \ T(\bm{p})= (2,1,1,1,1,2) \in \mathcal{I}_2$. 
}
\label{fig_rectangle}
\end{figure}

\begin{figure}[ht]
\centering
\includegraphics[height=4cm]{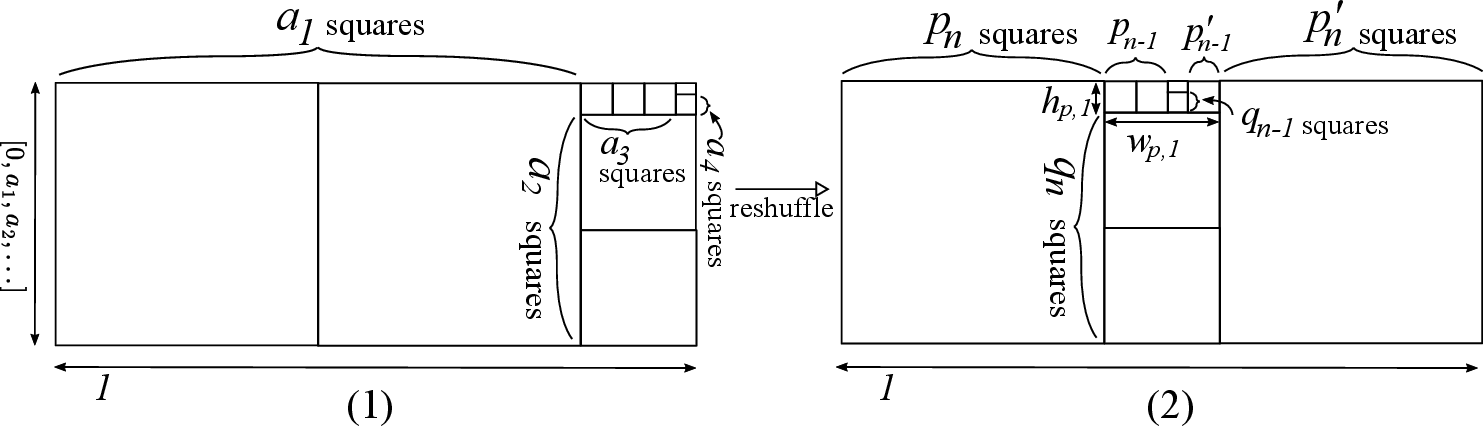}
\caption{
(1) Rectangle model for $[0, a_1, a_2,\dots]= [0,2,2,3,1, \dots]$. 
(2) Reshuffling squares when $[0, \overline{a_1, a_2,a_3, a_4}] = [0, \overline{1+1, 2,2+1, 1}]$. 
}
\label{fig_usefule_rectangle}
\end{figure}

\begin{definition} (Partitioned rectangle.) For $\bm{p} = (p_n, p_n', q_n, ..., p_1, p_1', q_1) \in \mathcal{I}_n$ 
we define a partitioned rectangle $\mathrm{rect}(\bm{p})$
as in Figure~\ref{fig_rectangle}. We start out with a rectangle of width $1$ 
and height $h_{\bm{p},0}= [0, \overline{p_n+p_n', q_n, ..., p_1 + p_1', q_1}]$. We then partition the rectangle into squares by the following procedure. First, we insert $p_n$ squares from the left. In the remaining rectangle, we insert $p_n'$ from the right and then $q_n$ from the bottom. We do the same for $p_{n-1}, p_{n-1}', q_{n-1}, \dots, p_{1}, p_{1}', q_{1}, p_{n}, p_{n}', q_{n}, \dots$, repeating the insertion pattern cyclically, infinitely many times. 
Rectangles for the example $\bm{p} = (1, 1, 2, 2, 1, 1) $ and $T(\bm{p})$ are illustrated in Figure~\ref{fig_rectangle}. 
\end{definition}

\begin{lem}
The partitioned rectangle $\mathrm{rect}(\bm{p})$ is well defined. 
\end{lem}

\begin{proof}
    We introduce a useful tool for infinite continued fractions. (See also \cite{Knott}.) 
    We define a rectangle whose width is $1$ and whose height is $[0, a_1, a_2,\dots]$ for $a_i \in {\mathbb N}$. 
    Then it is possible to iteratively  
    fill in $a_1, a_2,\dots$ squares as in Figure~\ref{fig_usefule_rectangle}(1). 
    Suppose that $[0, \overline{a_1, a_2, \dots , a_ {2n}}]= [0, \overline{p_n+ p_n', q_n, \dots, p_1+ p_1', q_1}]$. 
    We reshuffle the squares such that $p_i$ squares are filled from the left and $p_i'$  squares are filled  from the right 
    (see Figure~\ref{fig_rectangle}). This shows that the partition into squares for $\boldsymbol{p} \in \mathcal{I}_n$ is well defined. 
\end{proof}

The values $w_{\bm{p},0}$, $h_{\bm{p},0}$ can be thought of as the widths and heights of the rectangles obtained when we iteratively delete outside squares as in Figure~\ref{fig_rectangle}. The values are indicated in the picture. 

\begin{theorem}
\label{thm_expansion}
For $ \bm{p}=  (p_n, p_n', q_n, \dots, p_1, p_1', q_1) \in \mathcal{I}_n$ 
 the  Perron-Frobenius eigenvalue $\lambda_{\bm{p}}$ of $M_{\bm{p}}$ and its  eigenvector $\bm{v}>\bm{0}$ 
are given by 
$$\lambda_{\bm{p}}= \frac{1}{w_{\bm{p},n}}\ \hspace{2mm}\mbox{and}\ \hspace{2mm}
\bm{v} = \left(\begin{smallmatrix}  s_{\bm{p}} \\ h_{\bm{p},0} \\ 1- s_{\bm{p}} \end{smallmatrix}\right).$$ 
\end{theorem}

We call  $\bm{v}= \bm{v}_{\bm{p}}$  the {\em normalized eigenvector} with respect to $\lambda_{\bm{p}}$.

\begin{proof}
Recall that $T\colon\mathbb{N}_0^{3n} \to \mathbb{N}_0^{3n}$ is the shift 
    as in Section~\ref{section_Introduction}. 
    For  $\bm{p} \in \mathcal{I}_n$ 
    we define scaling factors $\lambda_{\bm{p}, i} \colonequals w_{{\bm{p}},i} / w_{\bm{p}, i+1}$ for $i \in \mathbb{N}_0$. 
    The scaling factors fulfill the property $\prod_{i=0}^{n-1}\lambda_{\bm{p}, i} = 1/w_{\bm{p}, n}$. 
    We will prove
    \begin{equation}
    \label{equation_statement-1}
        \lambda_{\bm{p}, 0}M_2^{-q_n}M_1^{-p_n}M_3^{-p_n'}
       \left(\begin{smallmatrix}s_{\bm{p}} \\ h_{\bm{p}, 0} \\ 1-s_{\bm{p}}\end{smallmatrix}\right)
        = 
       \left(\begin{smallmatrix}s_{T(\bm{p})} \\ h_{T(\bm{p}), 0} \\ 1-s_{T(\bm{p})}\end{smallmatrix}\right). 
    \end{equation}
    Using this, we can then inductively deduce the following statement:
    \begin{equation}
      \label{equation_statement-2}
    (\prod_{i=0}^{n-1}\lambda_{\bm{p}, i} )   (M_1^{p_n}M_3^{p_n'}M_2^{q_n} \cdots M_1^{p_1}M_3^{p_1'}M_2^{q_1})^{-1}
   \left(\begin{smallmatrix}s_{\bm{p}} \\ h_{\bm{p}, 0} \\ 1-s_{\bm{p}}\end{smallmatrix}\right)
    = \left(\begin{smallmatrix}s_{T^n(\bm{p})} \\ h_{T^n(\bm{p}), 0} \\ 1-s_{T^n(\bm{p})}\end{smallmatrix}\right)
    = \left(\begin{smallmatrix}s_{\bm{p}} \\ h_{\bm{p}, 0} \\ 1-s_{\bm{p}}\end{smallmatrix}\right) 
    \end{equation}
    The definitions of $w_{\bm{p}, i}$ and $h_{\bm{p}, i}$ are such that they line up with the lengths of the line segments in 
    $\mathrm{rect}(\bm{p})$ as in Figure~\ref{fig_rectangle}. Adding up the widths of all the squares on the left side, we get $s_{\bm{p}} (=  \sum_{i = 0}^\infty p_{-i}h_{\bm{p},i})$.
    Using Figure~\ref{fig_rectangle}, we observe that for 
    \begin{align*}
        \left(\begin{smallmatrix}  y_1 \\ y_2 \\ y_3 \end{smallmatrix}\right)
        \colonequals M_2^{-q_n}M_1^{-p_n}M_3^{-p_n'} 
       \left(\begin{smallmatrix}s_{\bm{p}} \\ h_{\bm{p},0} \\ 1-s_{\bm{p}}\end{smallmatrix}\right)
        =  M_2^{-q_n}
        \left(\begin{smallmatrix}s_{\bm{p}, 0}-p_nh_{\bm{p}, 0} \\ h_{\bm{p}, 0} \\ 1-s_{\bm{p}} - p_n'h_{\bm{p}, 0}\end{smallmatrix}\right) 
        =  
        \left(\begin{smallmatrix}s_{\bm{p}}-p_nh_{\bm{p}, 0} \\ h_{\bm{p}, 0} - q_n(1-(p_n+p_n')h_{\bm{p}, 0}) \\ 1-s_{\bm{p}} - p_n'h_{\bm{p}, 0}\end{smallmatrix}\right), 
    \end{align*}
    we have $w_{\bm{p}, 1} = 1- (p_n+ p_n') h_{\bm{p}, 0}= y_1 + y_3$ and $h_{\bm{p}, 1} =h_{\bm{p},0} - q_nw_{\bm{p}, 1} = y_2$.

    Remove $(p_n+ p_n')$ squares with height $h_{\bm{p}, 0}$ and $q_n$ squares with height $w_ {\bm{p}, 1}$ from $\mathrm{rect}(\bm{p})$. 
    If we then scale the remaining small rectangle by $\lambda_{\bm{p}, 0}= 1/w_{\bm{p},1} $, 
    its width becomes $1$ and the rectangle becomes a partitioned rectangle.  By moving all squares to the left, we see that its height must be 
    $[0, \overline{p_{n-1}+ p_{n-1}', q_{n-1}, \dots, p_1+ p_1', q_1, p_n+ p_n', q_n}] = h_ {T(\bm{p}), 0}$. 
    Its partition into squares then tells us that the resulting partitioned rectangle is $\mathrm{rect}(T(\bm{p}))$. 
    The value $y_1$ is the sum of the widths of all squares 
    sitting on the left of the small rectangle. 
    When scaling up $y_1$ by $\lambda_{\bm{p}, 0}$, 
   the value $\lambda_{\bm{p}, 0}y_1$ continues to be the sum of square widths. 
   This shows $\lambda_{\bm{p}, 0} y_1 = s_{T(\boldsymbol{p})}$. (See Figure \ref{fig_rectangle}.) 
    This proves  statement~(\ref{equation_statement-1}). 
    
    Statement~(\ref{equation_statement-2}) follows from applying statement ~(\ref{equation_statement-1}) $n$ times. 
    The value $w_{\bm{p}, n}^{-1} = \prod_{i=0}^{n-1}\lambda_{\bm{p}, i}$ then becomes the eigenvalue of the eigenvector 
    $\left(\begin{smallmatrix}s_{\bm{p}} \\ h_{\bm{p}, 0} \\ 1-s_{\bm{p}}\end{smallmatrix}\right)$ 
    of $M_{\bm{p}}$. Because the vector entries are all positive and $M_{\bm{p}}$ is Perron-Frobenius, 
     $w_{\bm{p}, n}^{-1}$ must be the Perron-Frobenius eigenvalue $\lambda_{\bm{p}}$ 
     by Theorem~\ref{thm_Perron-Frobenius}. 
\end{proof}

\begin{cor}
\label{cor_split-ratio}
The splitting ratio $s_{\bm{p}}$ can be written as follows. 
$$s_{\bm{p}} = \sum_{i = 0}^\infty p_{-i}h_{\bm{p}, i} 
= \frac{p_nh_{\bm{p}, 0} + p_{n-1}h_{\bm{p}, 1}+ \cdots +p_1h_{\bm{p},n-1}}{(p_n+p_n')h_{\bm{p}, 0} + (p_{n-1}+p_{n+1}')h_{\bm{p},1}+ \cdots +(p_1+p_1')h_{\bm{p}, n-1}}.$$
\end{cor} 

\begin{proof}
The split ratio $s_{\bm{p}}$ can be interpreted as a ratio dividing the width of the partitioned rectangle in two parts. 
Since the partitioned rectangle $\mathrm{rect}(\bm{p})$ is self-similar, it contains a 
rectangle that after rescaling by the factor $\lambda_{\bm{p}}$ is partitioned and equal to $\mathrm{rect}(\bm{p})$. 
To calculate $s_{\bm{p}}$, we can therefore ignore the width of the small self-similar rectangle and only use the ratio in the statement instead. 
\end{proof}

 \begin{rem}
 \label{rem_quadratic-irrational}
 For $\bm{p} \in \mathcal{I}_n$ 
 the height $h_{\bm{p}, 0} $ is a quadratic irrational 
 since the continued fraction expansion is eventually periodic. 
 One can prove inductively that the width $w_{\bm{p}, j}$ is a quadratic irrational for each $j \in {\mathbb N}_0$. 
 Thus $\lambda_{\bm{p}} = w_{\bm{p}, n}^{-1}$ 
 is also a quadratic irrational. 
 \end{rem}


\begin{cor}
\label{cor_same-dilatation}
Let $ \bm{p}=  (p_n, p_n', q_n, \dots, p_1, p_1', q_1)  \in \mathcal{I}_n $ and 
 $\bm{t}=  (t_n, t_n', u_n, \dots, t_1, t_1', u_1) \in \mathcal{I}_n$. 
If  $p_i+ p_i' = t_i+ t_i'$ and $q_i= u_i$ hold for all $i \in \{1, \dots, n\}$,  
then we have the following. 
\begin{enumerate}
\item[(1)] 
$\lambda_{\bm{p}}= \lambda_{\bm{t}}$. 

\item[(2)] 
If $\bm{t}= f(\bm{p})$, then 
$s_{\bm{p}}+ s_{f(\bm{p})} = 1$, 
where $f: {\mathbb N}_0^{3n}  \rightarrow {\mathbb N}_0^{3n} $ is the flip. 

\item[(3)] 
If $(p_n, p_{n-1}, \dots, p_1)  \prec (t_n, t_{n-1}, \dots, t_1) $, then 
$s_{\bm{p}}< s_{\bm{t}}$, 
where $\prec$ is the lexicographic ordering of ${\mathbb N}_0^n$. 
\end{enumerate}
\end{cor}

\begin{proof}
Claim (1) follows from Theorem~\ref{thm_expansion} 
since $w_{\bm{p}, n}= w_{\bm{t}, n}$ holds for all $n \in \mathbb{N}_0$. 
Exchanging $p_i$ and $p_i'$ for all $i \in \{1, ..., n\}$, flips the partitioned rectangle $\mathrm{rect}(\bm{p})$ horizontally. 
This means that  $ s_{f(\bm{p})}= 1- s_{\bm{p}}$. The proof of  (2) is done. 
For each $\bm{p} \in \mathcal{I}_n$ and all $i \in {\mathbb N}_0$, we have the property $w_{{\bm{p}, }i+1}< h_{\bm{p}, i}$. 
Using the definition of the partitioned rectangle, this implies claim (3). 
\end{proof} 

For a vector $\bm{v}= (v_i) \in  {\mathbb R}^n$, 
we  denote by $\bm{v}|_i$ 
the $i$-th coordinate $v_i$ of $\bm{v}$. 
When $M$ is an $n$ by $n$ square matrix, 
we also use the symbol $M \bm{v}|_i$ 
which returns the $i$-th coordinate of the vector $M \bm{v} $. 

\begin{cor}
\label{cor_PF-eivenvector-13}
For $\bm{p} \in \mathcal{I}_n$ 
let $\bm{v}>\bm{0}$ be a Perron-Frobenius eigenvector of $M_{\bm{p}}$. 
Then  $\bm{v}|_1 = \bm{v}|_3$ holds if and only if $\bm{p}$ is symmetric. 
\end{cor}

\begin{proof}
Corollary~\ref{cor_same-dilatation}(2)(3) implies that 
$s_{\bm{p}} =\frac{1}{2}$ holds if and only if $f(\bm{p}) = \bm{p}$ holds; i.e., $\bm{p}$ is symmetric. 
By Theorem~\ref{thm_expansion} 
the Perron-Frobenius eigenvector $\bm{v}= \bm{v}_{\bm{p}}$ satisfies the desired property. 
\end{proof}

\begin{ex}
Let us apply Theorem~\ref{thm_expansion} and Corollary~\ref{cor_split-ratio} to compute $s_{\bm{p}}$ 
and $\lambda_{\bm{p}}$. 
\begin{enumerate}
\item[(1)] 
Let $\bm{p}= (p, p', q) \in \mathcal{I}_1$. 
Then  $h_{\bm{p}, 0}= [0, \overline{p+p', q}\hspace{1mm}]$. 
We have 
$$s_{\bm{p}}=  \frac{p h_{\bm{p}, 0}}{(p+p') h_{\bm{p}, 0}}
=\frac{p}{p+p'},  \hspace{5mm} 
\lambda_{\bm{p}}= \frac{1}{w_{\bm{p}, 1}}=\frac{1}{1-(p+p') h_{\bm{p}, 0}}.$$

\item[(2)] 
Let  $\bm{p}= (1,0,1,0,1,1) \in \mathcal{I}_2$. 
We have 
$h_{\bm{p}, 0}=  [0, \overline{1}\hspace{1mm}]= \frac{-1+ \sqrt{5}}{2}$, 
$w_{\bm{p}, 1}= 1- h_{\bm{p}, 0}$, $h_{\bm{p}, 1}= h_{\bm{p}, 0}-w_{\bm{p}, 1}$, 
and  $w_{\bm{p}, 2}= w_{\bm{p}, 1}-h_{\bm{p}, 1}$. 
Hence,   $s_{\bm{p}}$ and  $\lambda_{\bm{p}}$ are given by 
$$s_{\bm{p}}= \frac{h_{\bm{p}, 0}}{h_{\bm{p}, 0}+ h_{\bm{p}, 1}}= \frac{h_{\bm{p}, 0}}{3h_{\bm{p}, 0}-1}, \hspace{5mm} 
\lambda_{\bm{p}}= \frac{1}{w_{\bm{p}, 2}}= \frac{1}{2-3h_{\bm{p}, 0}}=\frac{7+ 3\sqrt{5}}{2}.$$

\end{enumerate}
\end{ex}

By a calculation we have the following lemma. 

\begin{lem}
\label{lem_matrix-computation}
Let $q \in {\mathbb N}$ and $p, p' \in {\mathbb N}_0$. 
Let $\bm{x}= \left(\begin{smallmatrix} x \\ y \\ z \end{smallmatrix}\right)>\bm{0}$. 
\begin{enumerate}
\item[(1)]
$M_1^p M_3^{p}M_2^q \bm{x}|_1 \le M_1^p M_3^{p}M_2^q \bm{x}|_3 $ 
if and only if $x \le z$. 

\item[(2)] 
Suppose that  
$p> p' \ge 0$. 
Then  
$M_1^p M_3^{p'} M_2^q \bm{x}|_1 > M_1^p M_3^{p'} M_2^q \bm{x}|_3 $ 
for any $\bm{x}>\bm{0}$. 

\item[(3)] 
Suppose that $  0 \le p< p'$. 
Then  
$M_1^p M_3^{p'} M_2^q \bm{x}|_1 < M_1^p M_3^{p'} M_2^q \bm{x}|_3 $ 
for any $\bm{x}>\bm{0}$. 
\end{enumerate}
\end{lem}

As a corollary of Lemma~\ref{lem_matrix-computation}, 
we immediately have the following result. 

\begin{cor}
\label{cor_matrix-computation}
If $\bm{x}= \left(\begin{smallmatrix} x \\ y \\ z \end{smallmatrix}\right)$ is a positive vector with $x \ne z$, then 
$ M_1^p M_3^{p'} M_2^q \bm{x}|_1 \ne  M_1^p M_3^{p'} M_2^q \bm{x}|_3$ 
for any $q \in {\mathbb N}$ and $p, p' \in {\mathbb N}_0$ (possibly $p= p')$. 
\end{cor}

\subsection{Pseudo-Anosov maps in the semigroup \texorpdfstring{$F_D= F(\sigma_1, \sigma_3, \sigma_2^{-1})$}{Lg}} 
\label{subsection_semigroup_1}

We write $h_1 = \sigma_1$, $h_3 = \sigma_3$ and $h_2= \sigma_2^{-1}$. 
For a map $h = h_{n_k} \cdots h_{n_1} \in F_D$ ($n_i  \in \{1,2,3\} $) 
we set $M_h\colonequals  M_{n_k} \cdots M_{n_1}$. 
The following is a well-known result.

\begin{prop}
\label{prop_braid-Penner}
The product $h = h_{n_k} \cdots h_{n_1} \in F_D$ 
is pseudo-Anosov 
if all $\sigma_1, \sigma_3$ and $\sigma_2^{-1}$ appear in the product at least once. 
In this case 
the dilatation $\lambda(h)$ of $h$ equals the Perron-Frobenius eigenvalue $\lambda(M_h)$. 
\end{prop}
For the convenience of the reader, we give an outline of the proof.  
We use a criterion by Bestvina-Handel algorithm \cite{BestvinaHandel95} to determine when a mapping class is pseudo-Anosov. 
We first  choose a finite graph $G \subset \Sigma_{0,5}$ 
that is homotopy equivalent to $\Sigma_{0,5}$ as in Figure~\ref{fig_track_n}(2).  
The graph $G$ has four vertices $1, \dots, 4$ and four loop edges, 
each of which encircles a puncture. 
Let $P$ be the set of four loop edges of $G$.

Given a mapping class $\psi \in \mathrm{MCG}(\Sigma_{0,5})$, 
one can pick an induced graph map $g: G \rightarrow G$ homotopic to $\psi$. 
We require that $g$ sends vertices to vertices, edges to edge paths and fulfills $g(P)= P$. 
(See \cite[Section~1]{BestvinaHandel95}.) 
We may suppose that $g$ has no backtracks; i.e., 
$g$ maps each oriented edge of $ G$ to an edge path which does not contain an oriented edge $e$ 
followed by the same edge $\overline{e}$ with the opposite orientation. 
This map $g$ 
defines a $3$ by $3$ {\em transition matrix} $M$ (with respect to the $3$ non-loop edges). 
 For $r, s \in \{1, 2 , 3\}$ the entry $M_{rs}$ is the number of times that the $g$-image of the $s$-th edge  
runs the $r$-th  edge in either direction. 
We say that 
$g: G \rightarrow G$ is  {\em efficient} if $g^n: G \rightarrow G$ has no backtracks for all $n  > 0$.

Notice that   $h_i$ for $i \in \{1,2,3\}$ 
 induces a graph map $g_i: G \rightarrow G$ which has no backtracks 
as shown  Figure~\ref{fig_track_n}(1)--(4). 
The transition matrix of $g_i$ is given by the matrix $M_i$ 
as   in Section~\ref{section_Introduction}.

\begin{figure}[htbp]
\begin{center}
\includegraphics[height=5.6cm]{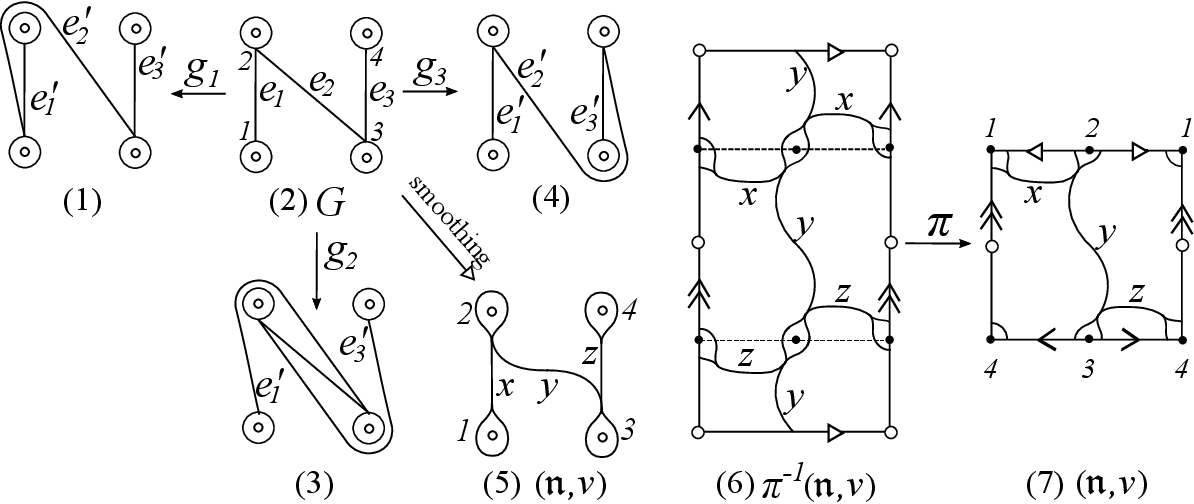}
\caption{
(1)--(4)  The graph maps $g_i: G \rightarrow G$. $e_j'\colonequals g_i(e_j)$. 
(5)  $(\mathfrak{n}, \bm{v}= \left(\begin{smallmatrix} x \\ y \\ z \end{smallmatrix}\right))$ in $\Sigma_{0,5}$. 
(6)(7)  The $2$-fold  branched cover 
$\pi \colon{\Sigma_{1,2}} \to \Sigma_{0,5}$. 
}
\label{fig_track_n}
\end{center}
\end{figure}

The composition  $ g_h: = g_{n_k} \cdots g_{n_1}: G \rightarrow G$ is an induced graph map of $h= h_{n_k} \cdots h_{n_1} $. 
(A priori, $g_h$ could have backtracks.) 
 We call  $k$  the length of the graph map $g_h$. 
 By induction on the length $k$, it can be shown that 
 $g_h: G \rightarrow G$ has no backtracks for any $h \in F_D$. 
 In particular, $g_h^n: G \rightarrow G$ has no backtracks for any $n >0$; i.e., $g_h: G \rightarrow G$ is efficient, 
 because $g_h^n$ is an induced graph map of $h^n \in F_D $. 
Since $ g_h: G \rightarrow G$ has no backtracks, 
the transition matrix with respect to the non-loop edges of $g_h$ is given by   $M_h$. 
If all $\sigma_1$, $\sigma_3$ and $\sigma_2^{-1}$ appear in the product $h$ at least once, then 
$M_h$ is Perron-Frobenius by Lemma~\ref{lem_Perron-Frobenius-p}. 
By the Bestvina-Handel algorithm \cite{BestvinaHandel95},
the two conditions ($g_h: G \rightarrow G$ is efficient,  and the transition matrix $M_h$ is Perron-Frobenius)  
ensure that $h$ is pseudo-Anosov with dilatation $\lambda(M_h)$.

\begin{remark}
\label{rem_suited-to}
Let $g_h: G \rightarrow G$ be an efficient graph map. We obtain a trivalent train track $\mathfrak{n}$ in $\Sigma_{0,5}$ (Figure~\ref{fig_track_n}(5)) by {\em graph smoothing} near the vertices of $G$.  
See \cite[Section~3.3]{BestvinaHandel95} for more details. 
Denote by $\bm{v}$ the Perron-Frobenius eigenvector  
of $M_ h$. We assign the weight $\bm{v}|_i$ (that is the $i$-th coordinate of $\bm{v}$) to the $i$-th branch and 
we obtain the measured train track  $(\mathfrak{n}, \bm{v})$ (also described in Figure~\ref{fig_track_n}(5)). 
This measured train track  $(\mathfrak{n}, \bm{v})$ 
is suited to the stable measured lamination of $h$ by \cite[Section~3.4]{BestvinaHandel95}. 
\end{remark}

\subsection{Pseudo-Anosov maps in the semigroup \texorpdfstring{$F_T= F(\delta_1, \delta_3, \delta_2^{-1})$}{Lg}}
\label{subsection_2-fold-branched-cover}

The union of curves $c_1 \cup c_2 \cup c_3$ (Figure~\ref{fig_t_track}(1)) fills the surface $\Sigma_{1,2}$. 
A construction of pseudo-Anosov maps by Penner \cite[Theorem~3.1]{Penner88} tells us that 
 the product of $\delta_1$, $\delta_3$ and $\delta_2^{-1}$  is pseudo-Anosov 
    if all the Dehn twists $\delta_1$, $\delta_3$ and $\delta_2^{-1}$ 
    appear in the product  at least once. 
Thus for each $\bm{p} \in \mathcal{I}_n$, 
the map $\Phi_{\bm{p}} \in F_T$ is pseudo-Anosov by the definition of  $\mathcal{I}_n$ (Definition~\ref{definition_In}).  
 The map 
$\phi_{\bm{p}} \in F_D$ is also pseudo-Anosov for each $\bm{p} \in \mathcal{I}_n$ 
by Proposition~\ref{prop_braid-Penner}.  
Additionally, each pseudo-Anosov map in $F_T$ (resp. $F_D$) is conjugate to 
$\Phi_{\boldsymbol{p}}$ (resp. $\phi_{\boldsymbol{p}}$) for some $\boldsymbol{p} \in \mathcal{I}_n$.
The link between the maps $\Phi_{\bm{p}}$ and $\phi_{\bm{p}}$ can be found in the following lemma.

\begin{lem}
\label{lem_same-dilatation}
For $\bm{p} \in \mathcal{I}_n$ let $\bm{v}>\bm{0}$ be an eigenvector 
for the Perron-Frobenius eigenvalue $\lambda_{\bm{p}}$ of  $M_{\bm{p}}$. 
Then  the measured train tracks 
$(\mathfrak{b}_L, \bm{v})$ in $\Sigma_ {0,5}$ and $(\mathfrak{b}, \bm{v})$ in $\Sigma_{1,2}$ 
defined in Section~\ref{section_Introduction} 
are suited to the stable measured laminations of  $\phi_{\bm{p}}  \in F_D$ and $\Phi_{\bm{p}} \in F_T$ respectively. 
Moreover, it holds 
$\lambda(\phi_{\bm{p}})= \lambda(\Phi_{\bm{p}})= \lambda_{\bm{p}}$, where $\lambda_{\bm{p}}$ is a quadratic irrational.
\end{lem}

\begin{proof}
By Remark~\ref{rem_suited-to} 
$(\mathfrak{n}, 2 \bm{v})$ is suited to the stable measured lamination of $\phi_{\bm{p}}$. 
Figure~\ref{fig_six_tracks} illustrates that $(\mathfrak{n}, 2 \bm{v})$ is equivalent to $(\mathfrak{b}_L, \bm{v})$. 
Therefore, $(\mathfrak{b}_L, \bm{v})$ is also suited to the stable measured lamination of $\phi_{\bm{p}}$.

\begin{figure}[ht]
\begin{center}
\includegraphics[height=3.2cm]{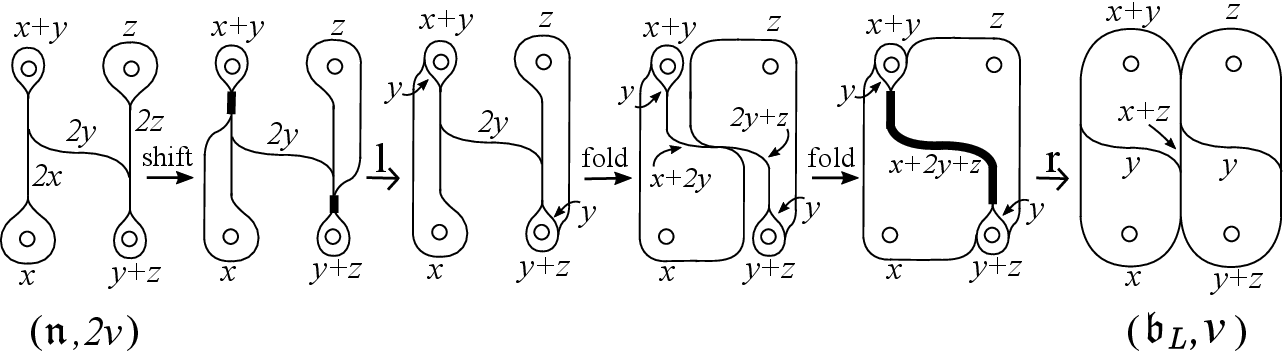}
\caption{
 $\stackrel{l}{\to}$ (resp. $\stackrel{r}{\to}$) denotes the left (resp. right) splittings at the highlighted large branches. 
 $(\mathfrak{n}, 2 \bm{v})$ is equivalent to $(\mathfrak{b}_L, \bm{v})$.}
\label{fig_six_tracks}
\end{center}
\end{figure}

We  regard  $\Sigma_{0,5}$ as the once punctured sphere 
with four marked points  $p_i$ ($i \in \{1, \dots, 4\}$). 
Consider a $2$-fold  branched cover 
$\pi \colon {\Sigma_{1,2}} \to \Sigma_{0,5}$ 
branched over the four marked points and 
induced by the hyperelliptic involution of $\Sigma_ {1,2}$, exchanging the two punctures.  
Notice that  $\delta_{j}\colonequals \delta_{c_j} \in \mathrm{MCG(\Sigma_{1,2})}$ 
is a lift of $\sigma_j \in \mathrm{MCG}(\Sigma_{0,5})$. Hence, 
$\Phi_{\bm{p}} \in F_T$ is a lift of $\phi_p \in F_D$. 
It follows that $\Phi_{\bm{p}}$  and $\phi_{\bm p}$ have the same dilatation. 
By Proposition~\ref{prop_braid-Penner} 
we have $\lambda(\phi_{\bm{p}})= \lambda_{\bm{p}}$. 
Thus $\lambda(\Phi_{\bm{p}})= \lambda(\phi_{\bm{p}})= \lambda_{\bm{p}}$. 
By Remark~\ref{rem_quadratic-irrational} $\lambda_{\bm{p}}$ is a quadratic irrational.

Let $\mathcal{F}^s$ and $\mathcal{F}^u$ be the stable and unstable foliations with respect to $\phi_{\bm{p}}$. 
The preimages $\pi^{-1}(\mathcal{F}^s)$ and $\pi^{-1}(\mathcal{F}^u)$ give the stable and unstable foliations with respect to $\Phi_{\bm{p}} $. 
Since  $p_i$ is a $1$-pronged singular point of $\mathcal{F}^s$ and $ \mathcal{F}^u$, 
the preimage $\pi^{-1}(p_i)$ is a regular point (i.e., a $2$-pronged point) 
of $\pi^{-1}(\mathcal{F}^s)$ and $ \pi^{-1}(\mathcal{F}^u)$.
Notice that $\pi^{-1}(\mathfrak{n})$ admits four bigons each of which contains a regular point $\pi^{-1}(p_i)$. 
See Figure~\ref{fig_track_n}(6)(7). 
Then  the measured train track $(\mathfrak{b}, \bm{v})$ in $\Sigma_{1,2}$ is obtained from $\pi^{-1}(\mathfrak{n},  \bm{v})$ 
by collapsing each bigon. 
As a result,  $(\mathfrak{b}, \bm{v})$ is suited to the stable measured lamination of $\Phi_{\bm{p}}$. 
\end{proof}

We will choose $(\mathfrak{b}, \lambda_{\bm{p}} \bm{v})$  
(resp. $(\mathfrak{b}_L, \lambda_{\bm{p}} \bm{v})$)  
as the start of the maximal splitting sequence  
in the proof of Theorem~\ref{thm_2-torus}  (resp. Theorem~\ref{thm_5-punctured-sphere-precise}).

\section{Agol cycles of pseudo-Anosov maps in  \texorpdfstring{$F_T$}{Lg}}
\label{section_FT}

The goal of this section is to prove Theorem~\ref{thm_2-torus}. 
To do this, we first  construct finite sequences of maximal splittings 
(Lemma~\ref{lem_step-maximal-splitting}, Proposition~\ref{prop_builidngblock}). 
Then  we concatenate some finite sequences to produce an Agol cycle of the pseudo-Anosov map $\Phi_{\bm{p}}$.

When $\bm{p}$ is symmetric, the normalized eigenvector $\bm{v}_{\bm{p}}$ with respect to $\lambda_{\bm{p}}$ fulfills 
$\bm{v}_{\bm{p}}|_1= \bm{v}_{\bm{p}}|_3$ (Corollary~\ref{cor_PF-eivenvector-13}). 
This extra symmetry gives simpler maximal splitting sequences. Hence, the measured train tracks with symmetric 
weights (i.e. $x=z$) and 
asymmetric weights (i.e. $x \ne z$) will be treated differently in the following lemma.

\begin{lem}
\label{lem_step-maximal-splitting}
Let $q \in {\mathbb N}$ and $p, p' \in {\mathbb N}_0$. 
Let  $\bm{x}= \left(\begin{smallmatrix} x \\ y \\ z \end{smallmatrix}\right) >\bm{0}$. 
\begin{enumerate}
\item[(1)] 
Suppose that $p > 0$. 
Then  
$$(\mathfrak{b}, M_1^{p-1} M_3^{p-1} M_2^q \bm{x}) = 
\left\{
\begin{array}{ll}
({\delta_1^{-1} \delta_3^{-1}} \circ {\rs}) (\mathfrak{b}, M_1^p M_3^p M_2^q \bm{x}) & \mbox{if}\ x = z,\\
(\delta_1^{-1} \delta_3^{-1} \circ {\rs^2}) (\mathfrak{b}, M_1^p M_3^p M_2^q \bm{x}) & \mbox{if}\ x \ne z.
\end{array}
\right.$$

\item[(2)] 
Suppose that $p > p' \ge 0$. 
Then  
$$(\mathfrak{b}, M_1^{p-1} M_3^{p'} M_2^q \bm{x}) = 
(\delta_1^{-1}  \circ \rs) (\mathfrak{b}, M_1^p M_3^{p'} M_2^q \bm{x}).$$

\item[(3)] 
Suppose that $0 \le p < p'$. 
Then  
$$(\mathfrak{b}, M_1^{p} M_3^{p'-1} M_2^q \bm{x}) = 
(\delta_3^{-1}  \circ \rs) (\mathfrak{b}, M_1^p M_3^{p'} M_2^q \bm{x}) .$$ 

\item[(4)] 
$(\mathfrak{b},  M_2^{q-1} \bm{x}) = 
\left\{
\begin{array}{ll}
({\delta_2} \circ {\ls^2}) (\mathfrak{b},  M_2^q \bm{x}) & \mbox{if}\ x = z,\\
({\delta_2} \circ {\ls^3}) (\mathfrak{b},  M_2^q \bm{x}) & \mbox{if}\ x \ne z.
\end{array}
\right.
$
\end{enumerate}
\end{lem}

\begin{proof}
A calculation 
$M_2^q \bm{x}= 
\left(\begin{smallmatrix} x \\ qx+ y+ qz \\ z \end{smallmatrix}\right)$ 
shows that  $\bm{x}|_1= M_2^q \bm{x}|_1 = x$ and $\bm{x}|_3= M_2^q \bm{x}|_3 = z$. 
For the proof of claims (1)--(4),   
it is suffices to prove them for $q=1$. 
In fact, 
once we prove claims (1)--(4) for $q=1$, 
we can apply them  to the positive vector $\bm{x}'= M_2^{q-1} \bm{x}$. 

We have 
$M_1^p M_3^p M_2 \left(\begin{smallmatrix} x \\ y \\ z \end{smallmatrix}\right)  
= \left(\begin{smallmatrix} x+ py' \\ y' \\ py' + z \end{smallmatrix}\right)$, 
where $y'= x+ y + z$. 
The measured train track  $(\tau_0, \mu_0) \colonequals (\mathfrak{b}, M_1^p M_3^p M_2 \bm{x}) $ 
has two large branches with  weights $x+ (p+1)y'$ and $(p+1)y'+z$. 
We first consider the case $x \ne z$. 
We may suppose that $x<z$. 
Applying $2$  maximal splittings (see Figure~\ref{fig_right_split}), we obtain 
$2$ right maximal splittings 
$$(\tau_0, \mu_0)= (\mathfrak{b}, M_1^p M_3^p M_2 \bm{x}) \rs^2 (\tau_2, \mu_2)=   \delta_1 \delta_3 (\mathfrak{b}, M_1^{p-1} M_3^{p-1} M_2 \bm{x}).$$
In other words, 
$(\mathfrak{b}, M_1^{p-1} M_3^{p-1} M_2 \bm{x}) = ({\delta_1^{-1} \delta_3^{-1}}  \circ {\rs^2}) (\mathfrak{b}, M_1^p M_3^p M_2 \bm{x})$. 
This gives claim (1) when $x < z$.

\begin{figure}[ht]
\centering
\includegraphics[height=4.5cm]{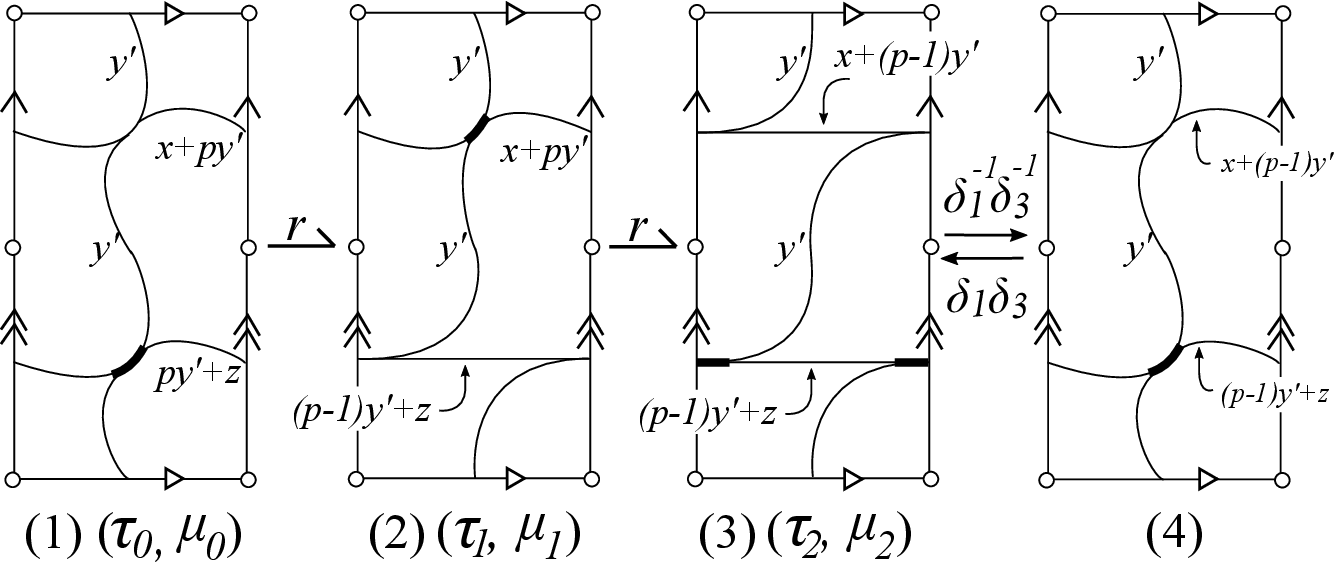}
\caption{
Proof of Lemma~\ref{lem_step-maximal-splitting}(1) 
when $x < z$. 
(1) $(\mathfrak{b}, M_1^p M_3^p M_2 \bm{x})$. 
(4) $(\mathfrak{b}, M_1^{p-1} M_3^{p-1} M_2 \bm{x})$. 
}
\label{fig_right_split}
\end{figure}

In the case $x = z$, 
 the two large branches of the measured train track $(\mathfrak{b}, M_1^p M_3^p M_2 \bm{x}) $  
have the same maximal weight. 
Applying the maximal splitting, 
we obtain the right maximal splitting 
$$(\tau_0, \mu_0)= (\mathfrak{b}, M_1^p M_3^p M_2 \bm{x}) \rs (\tau_1, \mu_1)=  \delta_1 \delta_3 (\mathfrak{b}, M_1^{p-1} M_3^{p-1} M_2 \bm{x}).$$
This completes the proof of claim (1).

We turn to claim (2). 
Suppose that $p> p' \ge 0$. 
We have 
$M_1^p M_3^{p'} M_2 \left(\begin{smallmatrix} x \\ y \\ z \end{smallmatrix}\right) 
= \left(\begin{smallmatrix} x+ py' \\ y' \\ p'y' + z \end{smallmatrix}\right)$, 
where $y' = x + y + z$. 
By a calculation we have 
$M_1^p M_3^{p'} M_2 \bm{x}|_1 =  x+py' > M_1^p M_3^{p'} M_2 \bm{x}|_3 = p'y'+z$.  
The measured train track $(\tau_0, \mu_0)\colonequals (\mathfrak{b}, M_1^p M_3^{p'}M_2 \bm{x})$ has  two large branches with  weights 
$x+ (p+1)y'$ and $(p'+1)y'+z$. 
Applying 
a maximal splitting (see Figure~\ref{fig_right_split_d1}), we obtain a right maximal splitting  
$$(\tau_0, \mu_0) =(\mathfrak{b}, M_1^p M_3^{p'} M_2 \bm{x}) \rs (\tau_1, \mu_1) = \delta_1 (\mathfrak{b}, M_1^{p-1} M_3^{p'} M_2 \bm{x}). $$
The proof of claim (2) is done. 
One can prove claim (3) in a similar way. 
\begin{figure}[ht]
\centering
\includegraphics[height=4.5cm]{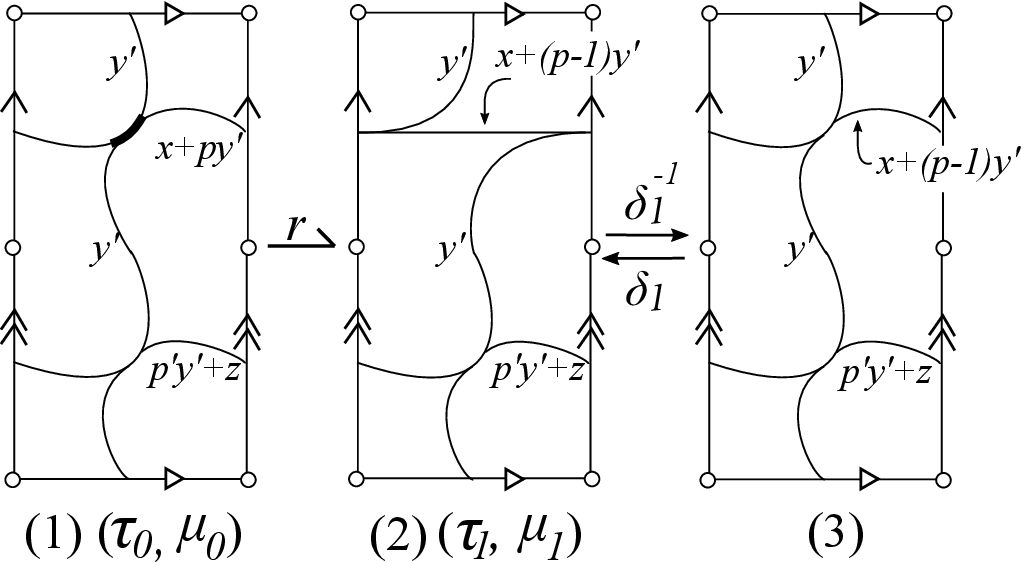}
\caption{
Proof of Lemma~\ref{lem_step-maximal-splitting}(2). 
(1) $(\mathfrak{b}, M_1^p M_3^{p'} M_2 \bm{x})$. 
(3) $(\mathfrak{b}, M_1^{p-1} M_3^{p'} M_2 \bm{x})$.}
\label{fig_right_split_d1}
\end{figure}

We now prove claim (4). 
We set $(\tau_0, \mu_0)= (\mathfrak{b}, M_2 \bm{x}=  \left(\begin{smallmatrix} x \\ x+y+z \\ z \end{smallmatrix}\right) )$. 
Consider the case $x \ne z$. 
We may suppose that  $x < z$. 
Applying $3$ 
maximal splittings  (see Figure~\ref{fig_left_split}), we obtain 3 left maximal splittings  
$$(\tau_0, \mu_0) = (\mathfrak{b}, M_2 \bm{x}) \ls (\tau_1, \mu_1) \ls (\tau_2, \mu_2) \ls  (\tau_3, \mu_3) = \delta_2^{-1} (\mathfrak{b},  \bm{x}).$$ 
This gives claim (4) for $x  < z$.

\begin{figure}[ht]
\centering
\includegraphics[height=4.5cm]{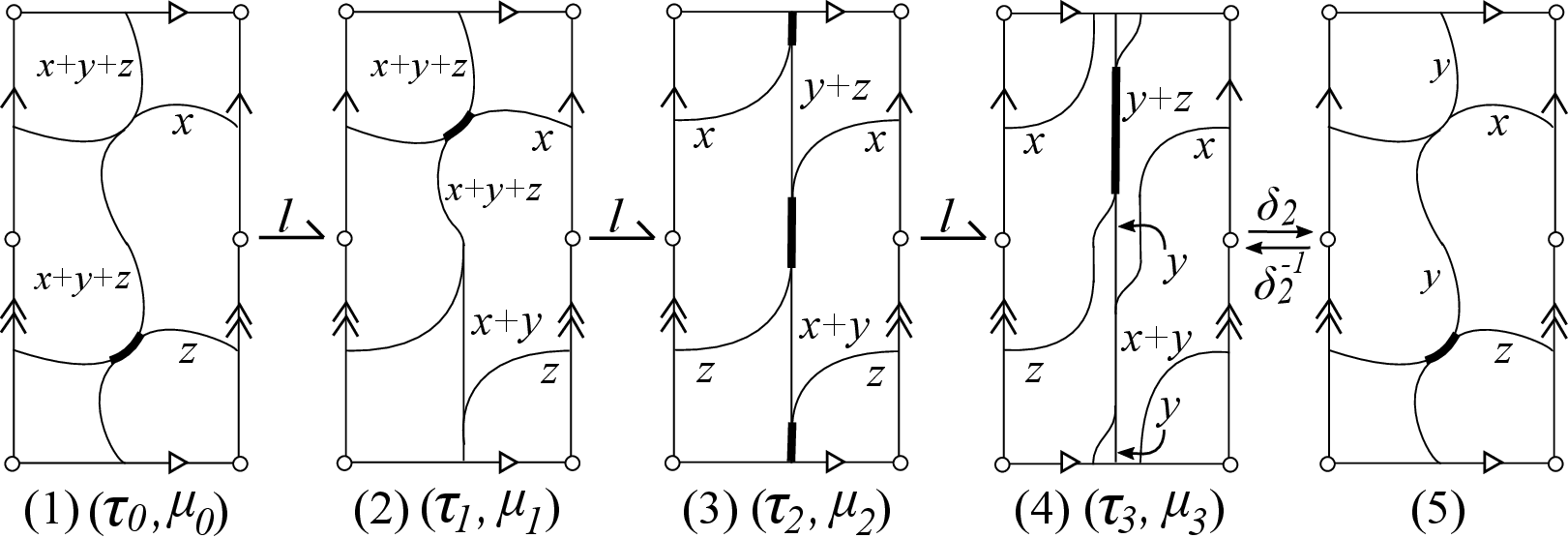}
\caption{
Proof of Lemma~\ref{lem_step-maximal-splitting}(4) when $x < z$. 
(1) $(\mathfrak{b},  M_2 \bm{x})$. 
(5) $(\mathfrak{b},  \bm{x})$.} 
\label{fig_left_split}
\end{figure}

In the case $x = z$, 
the two large branches of $(\mathfrak{b},  M_2 \bm{x})$ 
have the same maximal weight. 
Applying 
$2$ maximal splittings, we obtain $2$ left maximal splittings 
$ (\mathfrak{b},  M_2 \bm{x}) \ls^2   \delta_2^{-1} (\mathfrak{b}, \bm{x})$. 
This completes the proof. 
\end{proof}

\begin{prop}
\label{prop_builidngblock}
Let $q \in {\mathbb N}$ and $p, p' \in {\mathbb N}_0$.  
Let $\bm{x}=\left(\begin{smallmatrix} x \\ y \\ z \end{smallmatrix}\right)>\bm{0}$. 
\begin{enumerate}
\item[(1)] 
(Symmetric case.) 
Suppose that $p>0$. Then  
$$(\mathfrak{b}, \bm{x}) = (\delta_2^q \delta_1^{-p} \delta_3^{-p} \circ \ls^{2q} \circ \rs^{p}) (\mathfrak{b}, M_1^p M_3^{p} M_2^q {\bm x}) \hspace{2mm} \mbox{if}\  x=z.$$

\item[(2)] 
(Asymmetric case.) 
Suppose that  $p+ p'>0$ (possibly $p= p'>0$).  Then  
$$(\mathfrak{b}, \bm{x}) = (\delta_2^q \delta_1^{-p} \delta_3^{-p'} \circ \ls^{3q} \circ \rs^{p+p'}) (\mathfrak{b}, M_1^p M_3^{p'} M_2^q {\bm x}) \hspace{2mm} \mbox{if}\  x \ne z.$$
\end{enumerate}
\end{prop}

\begin{proof} 
We first prove claim (2) in the special case $p= p'>0$. 
Applying Lemma~\ref{lem_step-maximal-splitting}(1) in the latter case $x \ne z$, we have 
\begin{equation}
\label{equation_first-step}
(\mathfrak{b}, M_1^{p-1} M_3^{p-1} M_2^q \bm{x}) = ((\delta_1 \delta_3)^{-1} \circ \rs^2) (\mathfrak{b}, M_1^p M_3^p M_2^q \bm{x}).
\end{equation}
Then  applying Lemma~\ref{lem_step-maximal-splitting}(1) again, we obtain 
\begin{eqnarray*}
(\mathfrak{b}, M_1^{p-2} M_3^{p-2} M_2^q \bm{x}) &= &((\delta_1 \delta_3)^{-1} \circ \rs^2) (\mathfrak{b}, M_1^{p-1} M_3^{p-1} M_2^q \bm{x})
\\
&=& ((\delta_1 \delta_3)^{-1} \circ \rs^2) \circ ((\delta_1 \delta_3)^{-1} \circ \rs^2) (\mathfrak{b}, M_1^p M_3^{p} M_2^q {\bm x}) \hspace{5mm} (\because (\ref{equation_first-step})) 
\\
&=& ((\delta_1 \delta_3)^{-2} \circ \rs^4) (\mathfrak{b}, M_1^p M_3^{p} M_2^q {\bm x}). 
\hspace{5mm} (\because \mbox{Lemma~\ref{lem_commute}}).
\end{eqnarray*}
Repeating this argument, we have 
\begin{equation}
\label{equation_r-splitting}
(\mathfrak{b}, M_2^q \bm{x}) = ((\delta_1 \delta_3)^{-p} \circ \rs^{2p}) (\mathfrak{b}, M_1^p M_3^p M_2^q \bm{x}).
\end{equation}
Applying Lemma~\ref{lem_step-maximal-splitting}(4) in the case $x \ne z$ repeatedly, 
we have  
\begin{equation}
\label{equation_l-splitting}
(\mathfrak{b}, \bm{x}) = (\delta_2^q \circ \ls^{3q})  (\mathfrak{b}, M_2^q \bm{x}).
\end{equation}
The above equalities (\ref{equation_r-splitting}) and (\ref{equation_l-splitting}) give us 
\begin{eqnarray*}
(\mathfrak{b}, \bm{x}) 
&=& (\delta_2^q \circ \ls^{3q})  (\mathfrak{b}, M_2^q \bm{x}) \hspace{5mm} (\because (\ref{equation_l-splitting})) 
\\
&=& (\delta_2^q \circ \ls^{3q})  \circ  ((\delta_1 \delta_3)^{-p} \circ \rs^{2p}) (\mathfrak{b}, M_1^p M_3^{p} M_2^q \bm{x}) \hspace{5mm} (\because (\ref{equation_r-splitting})) 
\\
&=&  (\delta_2^q \delta_1^{-p} \delta_3^{-p}  \circ \ls^{3q}  \circ \rs^{2p}) (\mathfrak{b}, M_1^p M_3^{p} M_2^q \bm{x}). \hspace{5mm} 
(\because \mbox{Lemma~\ref{lem_commute}},\    \delta_1 \delta_3 = \delta_3 \delta_1) 
\end{eqnarray*}
This is the desired equality in the case $p= p'$. 
Next we prove claim (2) in the general case. 
We may suppose that  $0 \le p < p'$. 
Applying Lemma~\ref{lem_step-maximal-splitting}(3) repeatedly, 
we have  
\begin{equation}
\label{equation_r-splitting-sigma3}
(\mathfrak{b}, M_1^p M_3^p M_2^q \bm{x}) = ( \delta_3^{-(p'-p)} \circ \rs^{p'-p}) (\mathfrak{b}, M_1^p M_3^{p'} M_2^q \bm{x}).
\end{equation}
This together with the equalities (\ref{equation_r-splitting}) and (\ref{equation_l-splitting}) implies that 
\begin{eqnarray*}
(\mathfrak{b}, \bm{x}) 
&=& (\delta_2^q \circ \ls^{3q})  (\mathfrak{b}, M_2^q \bm{x}) \hspace{5mm} (\because (\ref{equation_l-splitting})) 
\\
&=& 
(\delta_2^q \circ \ls^{3q})  \circ  ((\delta_1 \delta_3)^{-p} \circ \rs^{2p}) (\mathfrak{b}, M_1^p M_3^{p} M_2^q \bm{x}) \hspace{5mm} (\because (\ref{equation_r-splitting})) 
\\
&=& 
(\delta_2^q \circ \ls^{3q})  \circ  ((\delta_1 \delta_3)^{-p} \circ \rs^{2p}) \circ ( \delta_3^{-(p'-p)} \circ \rs^{p'-p}) (\mathfrak{b}, M_1^p M_3^{p'} M_2^q \bm{x}) 
\hspace{5mm} (\because (\ref{equation_r-splitting-sigma3}))
\\
&=& 
 (\delta_2^q \delta_1^{-p} \delta_3^{-p'} \circ \ls^{3q} \circ \rs^{p+p'}) (\mathfrak{b}, M_1^p M_3^{p'} M_2^q {\bm x}) 
\hspace{5mm}(\because \mbox{Lemma~\ref{lem_commute}}).
\end{eqnarray*}
The proof of claim (2) is done. 
For the proof of claim (1), 
we assume $x=z$ and use Lemma~\ref{lem_step-maximal-splitting}(1)(4). 
This completes the proof. 
\end{proof}

We are ready to prove Theorem~\ref{thm_2-torus}.

\begin{proof}[Proof of Theorem~\ref{thm_2-torus}]
Let $M_{\bm{p}}= M_1^{p_n} M_3^{p_n'} M_2^{q_n} \cdots M_1^{p_1} M_3^{p_1'} M_2^{q_1}$ 
be the Perron-Frobenius matrix associated with $\bm{p} \in \mathcal{I}_n$. 
For a Perron-Frobenius eigenvector $\bm{v}$ of $M_ {\bm{p}}$, 
we define positive vectors 
$\bm{x}^{(0)} \colonequals  \bm{v}$ and 
$\bm{x}^{(i)} \colonequals  M_1^{p_i} M_3^{p_i'} M_2^{q_i} \bm{x}^{(i-1)} $ for $i \in \{1, \dots, n\}$. 
Then  $\bm{x}^{(n)} = M_{\bm{p}} \bm{v}= \lambda_{\bm{p}} \bm{v} $.

Suppose that $\bm{p}$ is asymmetric. 
By Corollaries~\ref{cor_PF-eivenvector-13} and \ref{cor_matrix-computation}, 
we can inductively prove that 
$\bm{x}^{(i)}|_1 \ne \bm{x}^{(i)}|_3$ for all $i \in \{0, \dots, n\}$. 
Proposition~\ref{prop_builidngblock}(2) tells us that 
$$
(\mathfrak{b}, \bm{x}^{(i-1)}) = (\delta_2^{q_i} \delta_1^{-p_i} \delta_3^{-p_i'} \circ \ls^{3q_i} \circ \rs^{p_i+p_i'}) (\mathfrak{b}, \bm{x}^{(i)} ) 
\hspace{5mm}\mbox{for\ } i \in \{1, \dots, n\}. 
$$
By the above equality for $i = 1, 2$, we obtain 
\begin{eqnarray*}
(\mathfrak{b}, \bm{v} ) &=& (\delta_2^{q_1} \delta_1^{-p_1} \delta_3^{-p_1'} \circ \ls^{3q_1} \circ \rs^{p_1+p_1'}) (\mathfrak{b}, \bm{x}^{(1)})
\\
&=& 
 (\delta_2^{q_1} \delta_1^{-p_1} \delta_3^{-p_1'} \circ \ls^{3q_1} \circ \rs^{p_1+p_1'}) \circ 
 (\delta_2^{q_2} \delta_1^{-p_2} \delta_3^{-p_2'} \circ \ls^{3q_2} \circ \rs^{p_2+p_2'}) (\mathfrak{b}, \bm{x}^{(2)})
 \\
 &=& 
 (\delta_2^{q_1} \delta_1^{-p_1} \delta_3^{-p_1'}\delta_2^{q_2} \delta_1^{-p_2} \delta_3^{-p_2'} \circ 
  \ls^{3q_1} \circ \rs^{p_1+p_1'} \circ \ls^{3q_2} \circ \rs^{p_2+p_2'}) (\mathfrak{b}, \bm{x}^{(2)}). 
\end{eqnarray*}
Repeating this argument, we finally obtain 
$$(\mathfrak{b}, \bm{v} ) = 
(\Phi_{\bm{p}}^{-1} \circ   \ls^{3q_1} \circ \rs^{p_1+p_1'} \circ \dots \circ  \ls^{3q_n} \circ \rs^{p_n+p_n'}) (\mathfrak{b},  \lambda_{\bm{p}} \bm{v} = \bm{x}^{(n)} ).$$ 
This means that 
$$(\mathfrak{b}, \lambda_{\bm{p}} \bm{v}) \rs^{p_n+p_n'} \ls^{3q_n}  \cdots \rs^{p_1+p_1'} \ls^{3q_1} \Phi_{\bm{p}}(\mathfrak{b}, \bm{v} ),$$
which is an Agol cycle of $\Phi_{\bm{p}}$ with length 
$ \sum_{i=1}^n (p_i+ p_i' + 3  q_i)$.

Suppose that  $\bm{p}$ is symmetric. 
By Corollary~\ref{cor_PF-eivenvector-13} $\bm{v}|_1 = \bm{v}|_3$ holds. 
A calculation shows that 
$\bm{x}^{(i)}|_1 = \bm{x}^{(i)}|_3$ for all $i \in \{0, \dots, n\}$. 
Applying Proposition~\ref{prop_builidngblock}(1), we have 
\begin{equation}
\label{equation_maximal-splittings}
(\mathfrak{b}, \bm{x}^{(i-1)}) = (\delta_2^{q_i} \delta_1^{-p_i} \delta_3^{-p_i} \circ \ls^{2q_i} \circ \rs^{p_i}) (\mathfrak{b}, \bm{x}^{(i)} ) 
\hspace{5mm}\mbox{for\ } i \in \{1, \dots, n\}. 
\end{equation}
Putting the above equalities (\ref{equation_maximal-splittings}) for each $i \in \{1, \cdots, n\}$ together,
we can obtain 
$$(\mathfrak{b}, \bm{v} ) = 
(\Phi_{\bm{p}}^{-1} \circ   \ls^{2q_1} \circ \rs^{p_1} \circ \dots \circ  \ls^{2q_n} \circ \rs^{p_n}) (\mathfrak{b},  \lambda_{\bm{p}} \bm{v} ).$$ 
This gives an Agol cycle of $\Phi_{\bm{p}}$ with length 
$ \sum_{i=1}^n (p_i + 2  q_i)$. 
We finished the proof.
\end{proof}

\begin{ex}\label{ex_simple_agol_cycle}
We present 2 examples for Agol cycles and their total splitting numbers. 
Recall that  $\bm{v}_{\bm{p}}$ is  the  normalized eigenvector with respect to $\lambda_{\bm{p}}$. 
\begin{enumerate}
\item 
For ${\bm{p}}= (1,1,1) \in \mathcal{I}_1$ symmetric, 
we have $\bm{v}_{\bm{p}} =  \left(\begin{smallmatrix} x\\ y \\ x \end{smallmatrix}\right)$ for some $x,y >0$ and 
$M_{\bm{p}} \bm{v}_{\bm{p}} = M_1 M_3 M_2 \bm{v}_{\bm{p}} = \left(\begin{smallmatrix} 3x+y\\ 2x+y \\ 3x+y \end{smallmatrix}\right)$. 
Figure \ref{fig_ex_symmetry} illustrates an Agol cycle 
    $(\mathfrak{b}, \lambda_{\bm{p}} \bm{v}_{\bm{p}}) \rs\ \ls^2 \Phi_{\bm{p}}(\mathfrak{b}, \bm{v}_{\bm{p}})$ of $\Phi_{\bm{p}}=  \delta_1\delta_3\delta_2^{-1} $  
    with length $3$. 
The splitting number of each maximal splitting in the Agol cycle is exactly $2$. 
Hence, we have $N(\Phi_{\bm{p}})= 2 \cdot 3 = 6$.

\item 
For ${\bm{p}}= (1,2,1) \in \mathcal{I}_1$ asymmetric, 
$(\mathfrak{b}, \lambda_{\bm{p}} \bm{v}_{\bm{p}}) \rs^3\ \ls^3 \Phi_{\bm{p}}(\mathfrak{b}, \bm{v}_{\bm{p}})$ 
is an Agol cycle of $\Phi_{\bm{p}} = \delta_1 \delta_3^2 \delta_2^{-1}$ with length $6$ by Theorem~\ref{thm_2-torus}. 
The splitting number of each maximal splitting in the Agol cycle is $1$, 
except for the last maximal splitting $\ls$ 
with the splitting number $2$. 
(See Figure~\ref{fig_left_split}(3)(4).)
Hence, we have $N(\Phi_{\bm{p}}) =  7$. 
\end{enumerate}
\begin{figure}[ht]
\centering
\includegraphics[height=3.8cm]{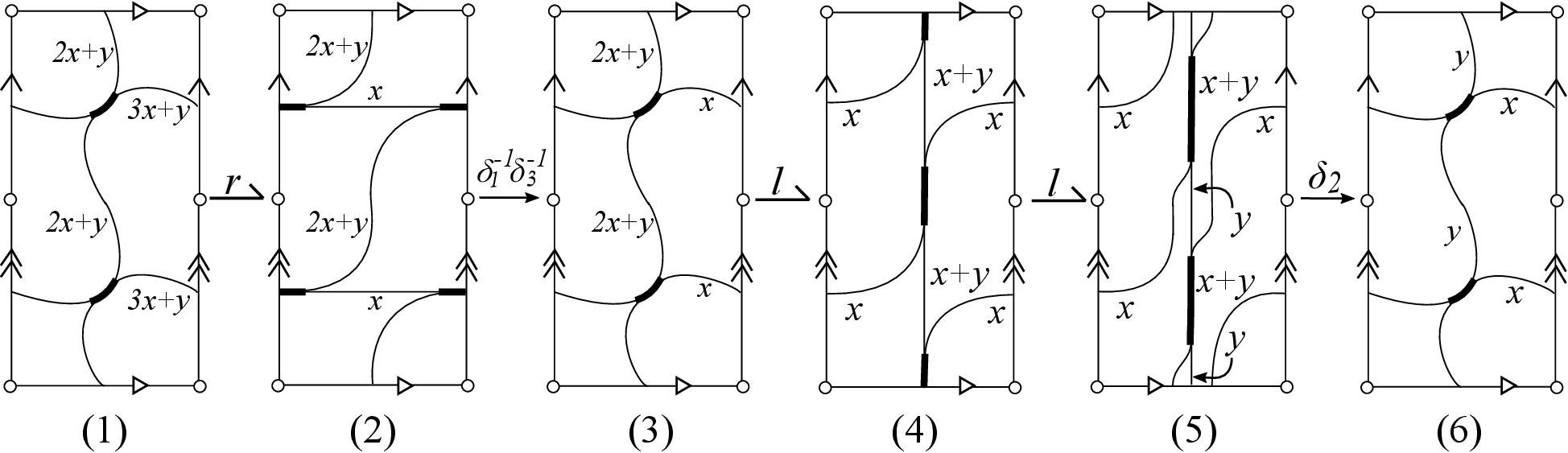}
\caption{An Agol cycle 
of $\Phi_{\bm{p}} $ for $\boldsymbol{p} =(1, 1, 1)$. 
(1) $(\mathfrak{b}, M_1 M_3 M_2\bm{v}_{\bm{p}}) $. 
(3) $(\mathfrak{b}, M_2 \bm{v}_{\bm{p}})$. 
(6) $(\mathfrak{b}, \bm{v}_{\bm{p}})$.} 
\label{fig_ex_symmetry}
\end{figure}
\end{ex}

\begin{thm}
\label{thm_total_FT}
For $ \bm{p}=  (p_n, p_n', q_n, \dots, p_1, p_1', q_1)  \in  \mathcal{I}_n$ 
the total splitting number of an Agol cycle of $\Phi_{\bm{p}}$ is given by 
$N(\Phi_{\bm{p}}) = \sum_{i=1}^n (p_i+ p_i' + 4q_i)$.  
\end{thm}

\begin{proof}
By Proposition~\ref{prop_builidngblock}(2) in the case of asymmetric weights, i.e. $x \ne z$, 
we have a finite sequence  
$(\mathfrak{b}, M_1^p M_3^{p'} M_2^q {\bm x}) \rs^{p+p'} \ls^{3q} \delta_1^p \delta_3^{p'} \delta_2^{-q}(\mathfrak{b}, \bm{x})$. 
The total splitting number of the finite sequence (Definition~\ref{definition_splitting-number}(2)) is $p+p'+ 4q$. 
The coefficient $4$ of $4q$ comes from the total splitting number of a finite sequence 
$(\mathfrak{b},  M_2^q \bm{x}) \ls^3 \delta_2^{-1} (\mathfrak{b},  M_2^{q-1} \bm{x})$ 
when $x \ne z$. 
See Figure~\ref{fig_left_split}. 
In the case of symmetric weights, i.e. $x = z$, 
Proposition~\ref{prop_builidngblock}(1) tells us that 
there exists a finite sequence 
$(\mathfrak{b}, M_1^p M_3^{p} M_2^q {\bm x}) \rs^{p} \ls^{2q} \delta_1^p \delta_3^{p} \delta_2^{-q}(\mathfrak{b}, \bm{x})$. 
Its total splitting number is $2(p+ 2q)= p+p + 4q$ 
since the splitting number of a maximal splitting in this finite sequence is exactly $2$. 

The weight  of  $(\mathfrak{b}, M_{\bm{p}} \bm{v}_{\bm{p}})$ 
 is given by $M_{\bm{p}}\bm{v}_{\bm{p}} =  M_1^{p_n} M_3^{p_n'} M_2^{q_n} \cdots M_1^{p_1} M_3^{p_1'} M_2^{q_1}\bm{v}_{\bm{p}} $. 
By the repetition of the above argument, 
we can prove that 
$N(\Phi_{\bm{p}}) = \sum_{i=1}^n (p_i+ p_i' + 4q_i)$. 
\end{proof}

\section{Agol cycles of pseudo-Anosov maps in  \texorpdfstring{$F_D$}{Lg}}
\label{section_FD}

We introduce positive integers $S_i(\bm{p})$ and $ A_i(\bm{p})$ for 
$ \bm{p}  \in  \mathcal{I}_n$ as follows. 
\begin{eqnarray*} 
S_i(\bm{p})=  p_i+2  \hspace{5mm}\mbox{and} \hspace{5mm} 
 A_i (\bm{p})= 
\left\{
\begin{array}{lll}
2p_i & \mbox{if}\ p_i'=0, 
\\
2p_i' & \mbox{if}\ p_i=0, 
\\
p_i+ p_i'+ 2  & \mbox{otherwise}.
\end{array}
\right.
\end{eqnarray*}
In this section, we prove the following result. 

\begin{thm}
\label{thm_5-punctured-sphere-precise}
For $ \bm{p}  \in  \mathcal{I}_n$ 
let  $\phi_{\bm{p}} \in F_D$  be the pseudo-Anosov map  
and  $M_{\bm{p}}$ be the Perron-Frobenius matrix  associated with $ \bm{p}$. 
Let ${\bm v} >\bm{0}$ be an eigenvector with respect to the Perron-Frobenius eigenvalue $\lambda_{\bm{p}}$ of  $M_{\bm{p}}$. 
Then  the Agol cycle length $\ell$ of $\phi_{\bm{p}}$ is 
$$  \ell= 
\left\{
\begin{array}{ll}
\sum_{i=1}^n (S_i (\bm{p})+ 2 q_i) 
& \mbox{if}\ \bm{p}\ \mbox{is symmetric}, 
\\
\sum_{i=1}^n (A_i (\bm{p})+ 3  q_i )
& \mbox{if}\ \bm{p}\ \mbox{is asymmetric}. 
\end{array}
\right.
$$
Moreover, starting with the measured train track $(\mathfrak{b}_0, \mu_0) = ( \mathfrak{b}_L, \lambda_{\bm{p}}\bm{v})$, 
a finite subsequence of  the maximal splitting sequence 
\begin{eqnarray*}
 (\mathfrak{b}_0, \mu_0) \rightharpoonup^{S_n(\bm{p})+ 2q_n}  \cdots \rightharpoonup^{S_1(\bm{p})+ 2q_1} (\mathfrak{b}_{\ell}, \mu_{\ell} ) 
& \mbox{if}& \hspace{-2mm} \bm{p}\ \mbox{is symmetric}, 
\\
(\mathfrak{b}_0, \mu_0)  \rightharpoonup^{A_n(\bm{p})+ 3q_n}  \cdots \rightharpoonup^{A_1(\bm{p})+ 3q_1}  (\mathfrak{b}_{\ell}, \mu_{\ell} ) 
& \mbox{if}& \hspace{-2mm} \bm{p}\ \mbox{is asymmetric}
\end{eqnarray*}
forms an Agol cycle of $\phi_{\bm{p}}$. 
The consecutive maximal splittings consist of the following left, right and mixed maximal splittings 
\begin{eqnarray*}
\rightharpoonup^{S_i(\bm{p})+ 2q_i} &=&  \rs\  \ls \  \rs^{p_i-1} \ \ls\  \rs \  \ls^{2q_i-1}, 
\\
\rightharpoonup^{A_i(\bm{p})+ 3q_i} &=& 
\left\{
\begin{array}{lll}
 \lrs\  \rs^{2p_i-1} \ \ls^{3q_i} & \mbox{if}\ p_i'=0, 
\\
\lrs\  \rs^{2p'_i-1} \  \ls^{3q_i} & \mbox{if}\ p_i=0, 
\\
\rs \  \ls\  \rs^{p_i+p_i'-2}\  \ls^2\  \rs\   \ls^{3q_i-1}& \mbox{otherwise}.
\end{array}
\right.
\end{eqnarray*}
\end{thm}

Figure~\ref{fig_braid_symmetry}(1) shows the measured train track $(\mathfrak{b}_L, \bm{x})$ that was defined in Section~\ref{section_Introduction}. 
Recall that the vector $\bm{x}$ reflects the weights of specific branches. 
Due to the switch condition, the weights on all remaining branches are determined. 
We introduce the measured train tracks 
$(\mathfrak{b}_R, \bm{x})$, 
$(\mathfrak{a}'_R, \bm{x})$ and 
$(\mathfrak{s}, \bm{x})$ in $\Sigma_ {0,5}$
as in Figure~\ref{fig_braid_symmetry}(2), (4) and (5) respectively. 
 Figure~\ref{fig_braid_symmetry}(3) gives the measured train track $\Delta(\mathfrak{a}'_R, \bm{x})$, 
 where $\Delta= \sigma_1 \sigma_2 \sigma_3 \sigma_1 \sigma_2 \sigma_1 \in \mathrm{MCG}(\Sigma_{0,5})$ is the $\pi$-rotation 
 (Figure~\ref{fig_braid_symmetry}(6)).

For $\phi_{\bm{p}}= \sigma_1^{p_n} \sigma_3^{p_n'} \sigma_2^{- q_n} \dots \sigma_1^{p_1} \sigma_3^{p_1'} \sigma_2^{- q_1} \in F_D$  
we call the product $ \sigma_1^{p_j} \sigma_3^{p_j'} \sigma_2^{- q_j}$ the ($j$-th) block of $\phi_{\bm{p}}$ and 
 say that the block is of type A (resp. A') if $p_j'= 0$ (resp. $p_j= 0$). 
Otherwise, we call it a type B block.

For the proof of Theorem~\ref{thm_5-punctured-sphere-precise} 
we consider each  block $\sigma_1^{p_j} \sigma_3^{p_j'} \sigma_2^{-q_j}$ of $\phi_{\bm{p}}$. 
The transition matrix induced by $\sigma_1^{p_j} \sigma_3^{p_j'} \sigma_2^{-q_j}$ is $M_1^{p_j} M_3^{p_j'} M_2^{q_j}$. 
Depending on the type of the block, consecutive maximal splittings of $(\mathfrak{b}_L, M_1^{p_j} M_3^{p_j'} M_2^{q_j} \bm{x})$ will result in different finite sequences. Figure~\ref{fig_braid_automaton} is the central tool in this paper. It illustrates how finite sequences of maximal splittings transition one measured train track into another. The details are given in Lemmas~\ref{lem_step-braid}, \ref{lem_degenerate-3} and \ref{lem_degenerate-1}. 
We will see that the concatenation of suitable finite sequences gives an Agol cycle of the pseudo-Anosov map $\phi_{\bm{p}}$.

\begin{figure}[ht]
\centering
\includegraphics[height=3.5cm]{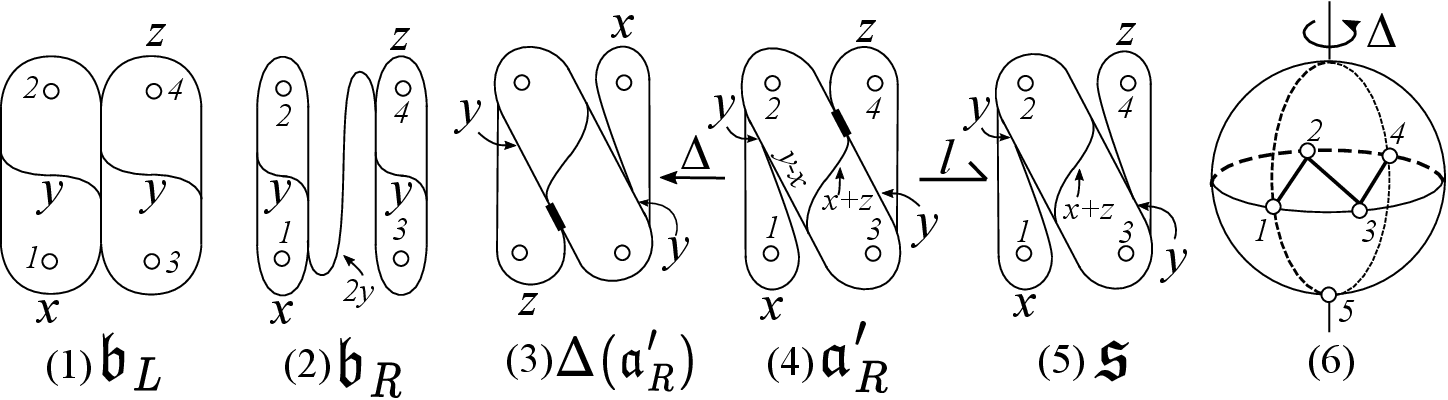}
\caption{
(1) $(\mathfrak{b}_L, \bm{x} )$, 
(2) $(\mathfrak{b}_R,  \bm{x})$, 
(3) $\Delta(\mathfrak{a}'_R,  \bm{x})$, 
(4) $ (\mathfrak{a}'_R,  \bm{x} )$, 
(5) $ (\mathfrak{s},  \bm{x})$ 
for $\bm{x}=  \left(\begin{smallmatrix} x \\ y \\ z \end{smallmatrix}\right)$. 
(6) $\Delta= \sigma_1 \sigma_2 \sigma_3 \sigma_1 \sigma_2 \sigma_1 \in \mathrm{MCG}(\Sigma_ {0,5})$. 
Figures (4)(5) illustrate a left maximal splitting 
$ (\mathfrak{a}'_R,  \bm{x} ) \ls  (\mathfrak{s},  \bm{x} )$ 
for $z < y$.}
\label{fig_braid_symmetry}
\end{figure}

\begin{figure}[ht]
\begin{center}
\includegraphics[height=6.5cm]{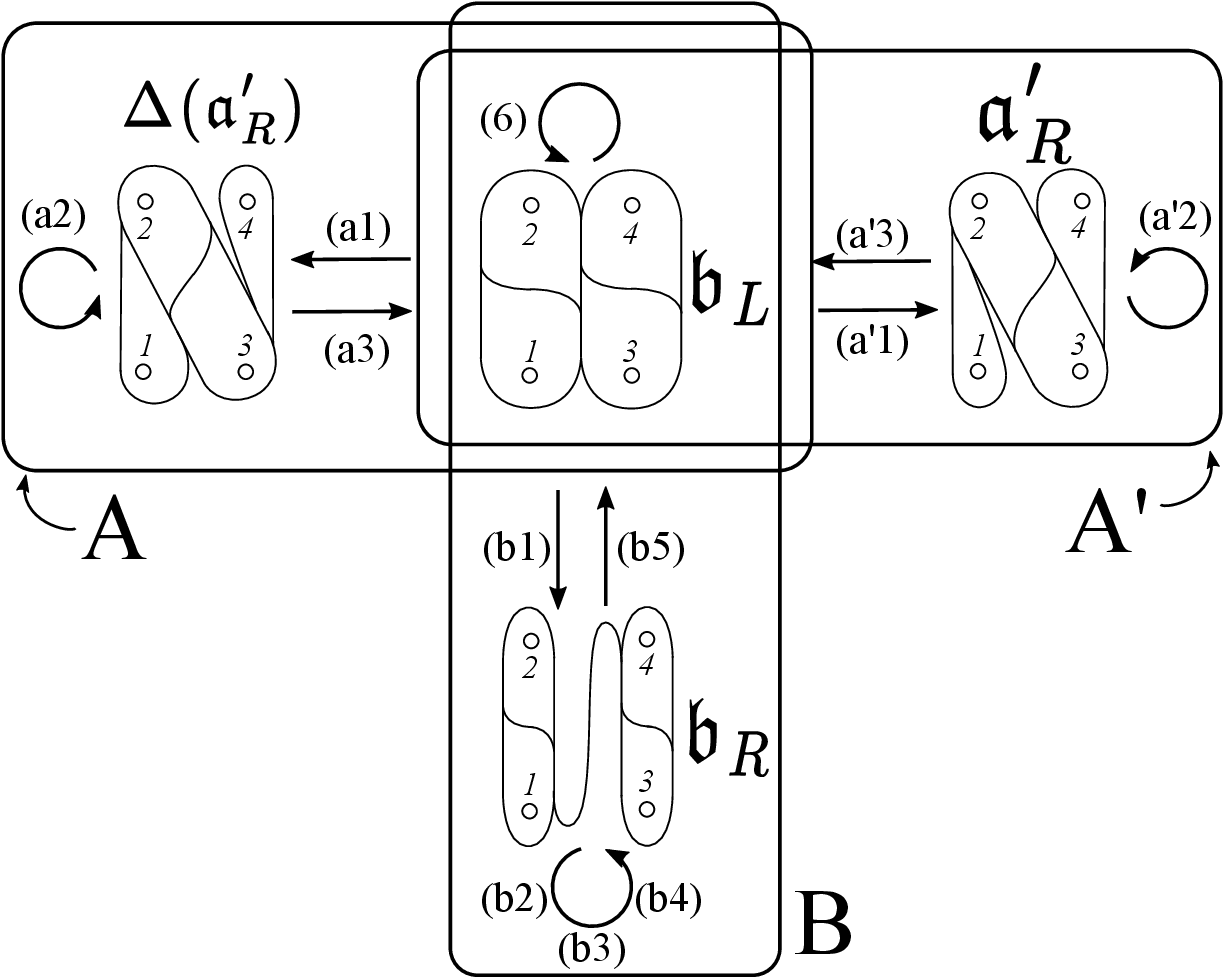}
\caption{``Automaton" illustrating how the train tracks move between topological types under the operations in Lemmas~\ref{lem_step-braid}, 
\ref{lem_degenerate-3} and \ref{lem_degenerate-1}. 
Box B displays Lemma~\ref{lem_step-braid}. 
Box A and A' display  Lemmas~\ref{lem_degenerate-1} and \ref{lem_degenerate-3} respectively. 
}
\label{fig_braid_automaton}
\end{center}
\end{figure}

\begin{lem}
\label{lem_step-braid}
Let $q \in {\mathbb N}$ and $p, p' \in {\mathbb N}_0$. 
Let $\bm{x}= \left(\begin{smallmatrix} x \\ y \\ z \end{smallmatrix}\right) >\bm{0}$.  
\begin{enumerate}
\item[(b1)] 
Suppose that $p, p' >0$. Then  
$$(\mathfrak{b}_R, M_1^{p-1} M_3^{p'-1} M_2^q \bm{x}) = (\sigma_1^{-1} \sigma_3^{-1} \circ \ls \circ \rs) (\mathfrak{b}_L, M_1^p M_3^{p'} M_2^q \bm{x}).$$

\item[(b2)] 
Suppose that $p>0$.  Then  
$$(\mathfrak{b}_R, M_1^{p-1} M_3^{p-1} M_2^q \bm{x}) = 
\left\{
\begin{array}{ll}
(\sigma_1^{-1} \sigma_3^{-1} \circ \rs) (\mathfrak{b}_R, M_1^{p} M_3^{p} M_2^q \bm{x}) & \mbox{if}\ x = z,\\
(\sigma_1^{-1} \sigma_3^{-1} \circ \rs^2) (\mathfrak{b}_R, M_1^{p} M_3^{p} M_2^q \bm{x}) & \mbox{if}\ x \ne z.
\end{array}
\right.$$

\item[(b3)] 
Suppose that $p > p' \ge 0$. 
Then  
$$(\mathfrak{b}_R, M_1^{p-1} M_3^{p'} M_2^q \bm{x}) = 
(\sigma_1^{-1}  \circ \rs) (\mathfrak{b}_R, M_1^p M_3^{p'} M_2^q \bm{x}).$$

\item[(b4)] 
Suppose that $0 \le p < p'$. 
Then  
$$(\mathfrak{b}_R, M_1^{p} M_3^{p'-1} M_2^q \bm{x}) = 
(\sigma_3^{-1}  \circ \rs) (\mathfrak{b}_R, M_1^p M_3^{p'} M_2^q \bm{x}) .$$

\item[(b5)]
$(\mathfrak{b}_L,  M_2^{q-1} \bm{x}) = 
\left\{
\begin{array}{ll}
(\sigma_2 \circ \ls \circ \rs \circ \ls) (\mathfrak{b}_R,  M_2^q \bm{x}) & \mbox{if}\ x = z,\\
(\sigma_2 \circ \ls^2 \circ \rs \circ \ls^2) (\mathfrak{b}_R,  M_2^q \bm{x}) & \mbox{if}\ x \ne z.
\end{array}
\right.
$

\item[(6)] 
$(\mathfrak{b}_L,  M_2^{q-1} \bm{x}) = 
\left\{
\begin{array}{ll}
(\sigma_2 \circ \ls^2) (\mathfrak{b}_L,  M_2^{q} \bm{x}) & \mbox{if}\ x = z,\\
(\sigma_2 \circ \ls^3) (\mathfrak{b}_L,  M_2^{q} \bm{x}) & \mbox{if}\ x \ne z.
\end{array}
\right.
$
\end{enumerate}
\end{lem}

\begin{proof}
 Figure~\ref{fig_rl13} shows that 
$ (\mathfrak{b}_L, M_1 M_3\bm{a}= \left(\begin{smallmatrix} a+b \\ b \\ b+c \end{smallmatrix}\right) ) \rs\  \ls 
\sigma_1 \sigma_3 (\mathfrak{b}_R, \bm{a})$ 
for $\bm{a}= \left(\begin{smallmatrix} a \\ b \\ c \end{smallmatrix}\right) >\bm{0}$. 
In other words, 
$(\mathfrak{b}_R, \bm{a}) = (\sigma_1^{-1} \sigma_3^{-1} \circ \ls \circ \rs) (\mathfrak{b}_L, M_1 M_3\bm{a})$. 
Choosing $\bm{a}= M_1^{p-1} M_3^{p'-1} M_2^q \bm{x}$ as a positive vector, 
we obtain claim (b1).

\begin{figure}[ht]
\begin{center}
\includegraphics[height=3.5cm]{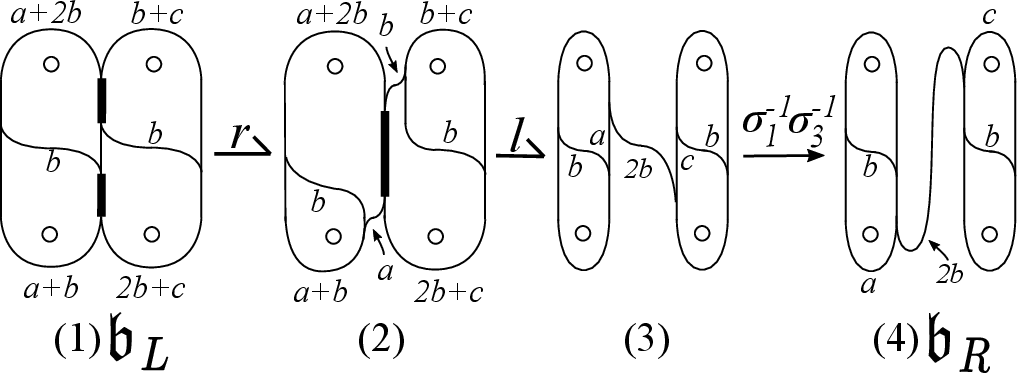}
\caption{
Proof of Lemma~\ref{lem_step-braid}(b1). 
(1) $(\mathfrak{b}_L, M_1 M_3 \bm{a})$. 
(4) $(\mathfrak{b}_R,  \bm{a})$.}
\label{fig_rl13}
\end{center}
\end{figure}

It is enough to prove the remaining claims when $q=1$. 
For  claim (b2), we set 
$(\mathfrak{b}_0, \mu_0) = (\mathfrak{b}_R, M_1^p M_3^p M_2 \bm{x})$. 
The proof is similar to that of Lemma~\ref{lem_step-maximal-splitting}(1). 
Figure~\ref{fig_rr13_splitting} illustrates the proof of (b2) when $x < z$. 
In the case $x = z$, 
the two large branches of  $(\mathfrak{b}_R, M_1^p M_3^p M_2 \bm{x})$  
have the same weight. 
(c.f. Figure~\ref{fig_rr13_splitting}(1).) 
Applying a maximal splitting, we obtain the right maximal splitting 
$ (\mathfrak{b}_R, M_1^p M_3^p M_2 \bm{x}) \rs   \sigma_1 \sigma_3 (\mathfrak{b}_R, M_1^{p-1} M_3^{p-1} M_2 \bm{x})$. 
This completes the proof of claim (b2).

\begin{figure}[htbp]
\begin{center}
\includegraphics[height=3.5cm]{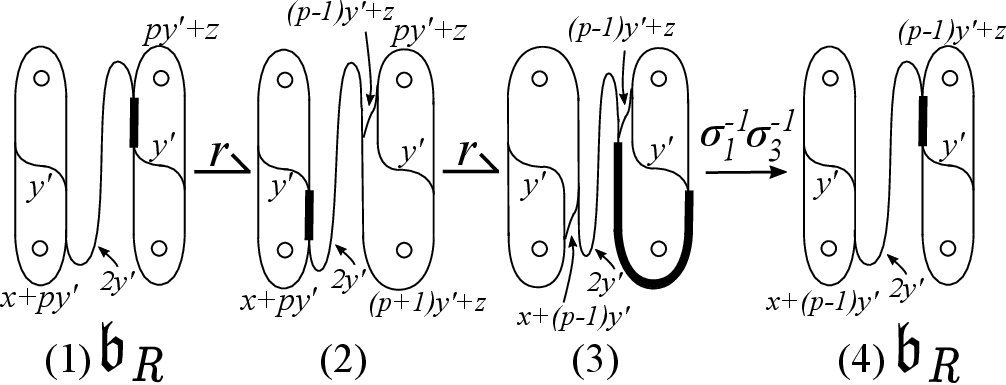}
\caption{
Proof of Lemma~\ref{lem_step-braid}(b2) when $x < z$. 
(1) $(\mathfrak{b}_R, M_1^{p} M_3^{p} M_2 \bm{x})$. 
(4) $(\mathfrak{b}_R, M_1^{p-1} M_3^{p-1} M_2 \bm{x})$.}
\label{fig_rr13_splitting}
\end{center}
\end{figure}

The proof of claim (b3) (resp. (b4)) is similar to that of Lemma~\ref{lem_step-maximal-splitting}(2) (resp. Lemma~\ref{lem_step-maximal-splitting}(3)) 
and we omit the proof.

Before proving claim (b5), we first prove claim (6). 
We consider the measured train track 
 $(\mathfrak{b}_0, \mu_0)= (\mathfrak{b}_L, M_2 \bm{x}=  \left(\begin{smallmatrix} x \\ x+y+z \\ z \end{smallmatrix}\right) )$ 
 when $x \ne z$. 
We may suppose that  $x < z$. 
Applying  $3$ maximal splittings (see Figure~\ref{fig_llls2}(1)--(4)), 
we have $3$ left maximal splittings 
\begin{equation}
\label{equation_3left-splittings}
(\mathfrak{b}_0, \mu_0) = (\mathfrak{b}_L, M_2 \bm{x}) \ls (\mathfrak{b}_1, \mu_1) = (\mathfrak{s}, M_2 \bm{x}) \ls 
(\mathfrak{b}_2, \mu_2) \ls  (\mathfrak{b}_3, \mu_3) = \sigma_2^{-1} (\mathfrak{b}_L,  \bm{x}).
\end{equation}
This gives claim (6) when $x  < z$.

We turn to  the case $x = z$. 
Applying $2$ maximal splittings, 
we obtain $2$ left maximal splittings 
$$
(\mathfrak{b}_0, \mu_0) = (\mathfrak{b}_L, M_2 \bm{x})   \ls (\mathfrak{b}_1, \mu_1) = (\mathfrak{s}, M_2 \bm{x})
  \ls (\mathfrak{b}_2, \mu_2)=  \sigma_2^{-1} (\mathfrak{b}_L, \bm{x}).
$$
This gives the proof of claim (6) when $x=z$.

We finally prove claim (b5). 
Consider the measured train track 
$(\mathfrak{b}_0, \mu_0) = (\mathfrak{b}_R, M_2 \bm{x})$ when $x \ne z$. 
We may suppose that $x < z$. 
Figures~\ref{fig_llls2}(1')--(3') and (2) show that 
$(\mathfrak{b}_0, \mu_0) =(\mathfrak{b}_R, M_2 \bm{x}) \ls^2 \  \rs (\mathfrak{s}, M_2 \bm{x})$. 
Taking the last two  maximal splittings from the finite sequence $(\ref{equation_3left-splittings})$, 
we have 
$ (\mathfrak{s}, M_2 \bm{x}) \ls^2  \sigma_2^{-1} (\mathfrak{b}_L,  \bm{x})$. 
Putting them together, we have 
$$(\mathfrak{b}_0, \mu_0) =(\mathfrak{b}_R, M_2 \bm{x}) \ls^2 \  \rs\ (\mathfrak{s}, M_2 \bm{x})\  \ls^2  \sigma_2^{-1} (\mathfrak{b}_L, \bm{x}).$$
This gives claim (b5) when $x < z$. 

\begin{figure}[ht]
\centering
\includegraphics[height=7cm]{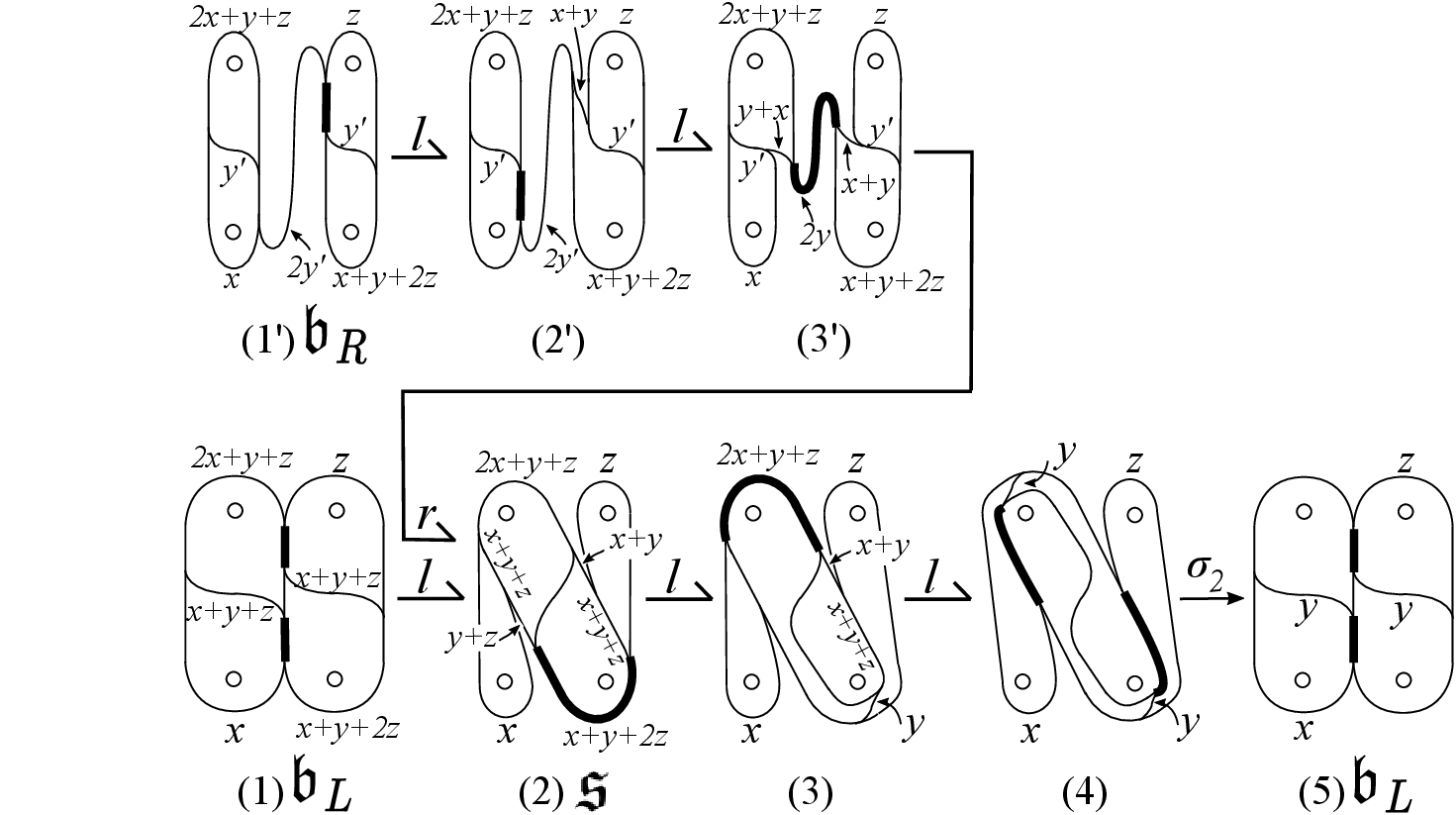}
\caption{(1)--(5) Proof of Lemma~\ref{lem_step-braid}(6) when $x < z$. 
(1')--(3')(2)--(5) Proof of Lemma~\ref{lem_step-braid}(b5) when $x < z$. 
}
\label{fig_llls2}
\end{figure}

In the case $x=z$, 
the measured train track 
$(\mathfrak{b}_R, M_2 \bm{x})$ has two large branches with maximal weight. This gives the finite sequence 
$(\mathfrak{b}_L, M_2 \bm{x}) \ls \  \rs\ (\mathfrak{s}, M_2 \bm{x})\  \ls  \sigma_2^{-1} (\mathfrak{b}_L, \bm{x})$. 
This completes the proof. 
\end{proof}

Let $(\mathfrak{b}_L, M_1^p M_3^p M_2^q \bm{x})$ be a measured train track, 
where  the measure $M_1^p M_3^p M_2^q \bm{x}$ is preceded by a type B block. By repeatedly applying the last lemma, we now compute the maximal splittings of $(\mathfrak{b}_L, M_1^p M_3^p M_2^q \bm{x})$.

\begin{prop}[Type B block]
\label{prop_typeB}
Let $p, p', q \in {\mathbb N}$. 
Let $\bm{x}= \left(\begin{smallmatrix} x \\ y \\ z \end{smallmatrix}\right) >\bm{0}$. 
\begin{enumerate}
\item[(1)] 
(Symmetric case.) 
$(\mathfrak{b}_L, \bm{x}) = (\sigma_2^q \sigma_1^{-p} \sigma_3^{-p} \circ  \rightharpoonup^{p+2+2q}) (\mathfrak{b}_L, M_1^p M_3^p M_2^q \bm{x})$ 
if $x = z$. 
The consecutive maximal splittings consist of the following left and  right maximal splittings 
$$\rightharpoonup^{p+2+2q}\ = \  \ls^{2q-1} \hspace{-2mm} \circ \rs \circ \ls \circ \rs^{p-1} \circ \ls \circ \rs.$$

\item[(2)] 
(Asymmetric case.) 
$(\mathfrak{b}_L, \bm{x}) = (\sigma_2^q \sigma_1^{-p} \sigma_3^{-p'} \circ  \rightharpoonup^{p+ p'+2+3q}) (\mathfrak{b}_L, M_1^p M_3^{p'} M_2^q \bm{x})$ 
if $x \ne z$, possibly $p= p'$. 
The consecutive maximal splittings consist of the following left and  right maximal splittings 
$$ \rightharpoonup^{p+ p'+2+3q} \ =\ \ls^{3q-1} \hspace{-2mm} \circ \rs \circ \ls^2 \circ \rs^{p+p'-2} \circ \ls \circ \rs.$$
\end{enumerate}
\end{prop}

\begin{proof}
We prove claim (2). 
Suppose that $x \ne z$. 
We may assume that $p< p'$. 
(The proof for the case $p \ge p'$ can be treated in the same manner.) 
We have 
\begin{eqnarray*}
(\mathfrak{b}_R, M_1^{p-1} M_3^{p'-1} M_2^q \bm{x}) &=& (\sigma_1^{-1} \sigma_3^{-1} \circ \ls \circ \rs) (\mathfrak{b}_L, M_1^p M_3^{p'} M_2^q \bm{x}) 
\hspace{2mm}  \mbox{(Lemma~\ref{lem_step-braid}(b1))}, 
\\
(\mathfrak{b}_R, M_1^{p-1} M_3^{p-1} M_2^q \bm{x}) &=& 
(\sigma_3^{-(p'-p)}  \circ \rs^{p'-p}) (\mathfrak{b}_R, M_1^{p-1} M_3^{p'-1} M_2^q \bm{x}) \hspace{2mm}  \mbox{(Lemma~\ref{lem_step-braid}(b4))}, 
\\
(\mathfrak{b}_R,  M_2^q \bm{x}) &=& 
((\sigma_1\sigma_3)^{-(p-1)} \circ \rs^{2p-2}) (\mathfrak{b}_R, M_1^{p-1} M_3^{p-1} M_2^q \bm{x})  \hspace{2mm}  \mbox{(Lemma~\ref{lem_step-braid}(b2))}, 
\\
(\mathfrak{b}_L, \bm{x}) &=& (\sigma_2^q \circ \ls^{3(q-1)} \circ \ls^2  \circ \rs \circ \ls^2)  (\mathfrak{b}_R,  M_2^q \bm{x}) \hspace{2mm} \mbox{(Lemma~\ref{lem_step-braid}(b5),(6))}. 
\end{eqnarray*}
By  the above equalities together with Lemma~\ref{lem_commute}, 
we obtain 
$$(\mathfrak{b}_L,   \bm{x}) =  
(\sigma_2^q \sigma_1^{-p} \sigma_3^{-p'} \circ \ls^{3q-1} \hspace{-2mm} \circ \rs \circ \ls^2 \circ \rs^{p+p'-2} \circ \ls \circ \rs) (\mathfrak{b}_L, M_1^p M_3^{p'} M_2^q \bm{x}).$$
This completes the proof of (2). 
The proof of claim (1) is left to the reader. 
\end{proof}

\begin{lem}
\label{lem_degenerate-3} 
Let $q, s \in {\mathbb N}$. 
Let $\bm{x}= \left(\begin{smallmatrix} x \\ y \\ z \end{smallmatrix}\right) >\bm{0}$. 
\begin{enumerate}
\item[(a'1)] 
$(\mathfrak{a}'_R, M_3^{s-1} M_2^q \bm{x}) = (\sigma_3^{-1} \circ \rs \circ  \lrs) (\mathfrak{b}_L, M_3^s M_2^q \bm{x})$. 

\item[(a'2)] 
$(\mathfrak{a}'_R, M_3^{s-1} M_2^q \bm{x}) = (\sigma_3^{-1} \circ \rs^2) (\mathfrak{a}'_R, M_3^s M_2^q \bm{x}) $. 

\item[(a'3)] 
$(\mathfrak{b}_L,  M_2^{q-1} \bm{x}) = ( \sigma_2 \circ \ls^3) (\mathfrak{a}'_R,  M_2^q \bm{x})$ if $x \ne z$. 
\end{enumerate}
\end{lem}

\begin{proof}
It is sufficient to prove the lemma when $q=1$. 
Consider the maximal splitting starting from $ (\mathfrak{b}_L, M_3^s M_2 \bm{x} = \left(\begin{smallmatrix} x \\ y' \\ sy'+z \end{smallmatrix}\right))$, 
where $y' = x+y+z$. 
\begin{figure}[ht]
\centering
\includegraphics[height=3.5cm]{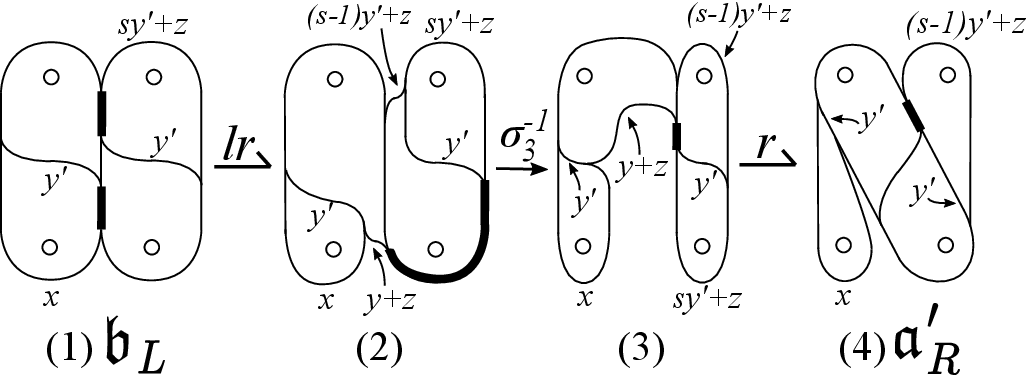}
\caption{Proof of Lemma~\ref{lem_degenerate-3}(a'1). 
(1) $ (\mathfrak{b}_L, M_3^s M_2 \bm{x})$.\\
(4) $(\mathfrak{a}'_R, M_3^{s-1} M_2 \bm{x})$.}
\label{fig_mix_s3}
\end{figure}
Figure~\ref{fig_mix_s3} shows that 
$$
(\mathfrak{a}'_R, M_3^{s-1} M_2 \bm{x}) 
= (\rs \circ \sigma_3^{-1} \circ  \lrs) (\mathfrak{b}_L, M_3^s M_2 \bm{x}) 
= (\sigma_3^{-1} \circ \rs \circ \lrs) (\mathfrak{b}_L, M_3^s M_2 \bm{x}). 
$$
The proof of claim (a'1) is done. 
For the proof of claim (a'2), 
see Figure~\ref{fig_mix_s3-part2}. 
\begin{figure}[ht]
\centering
\includegraphics[height=3.5cm]{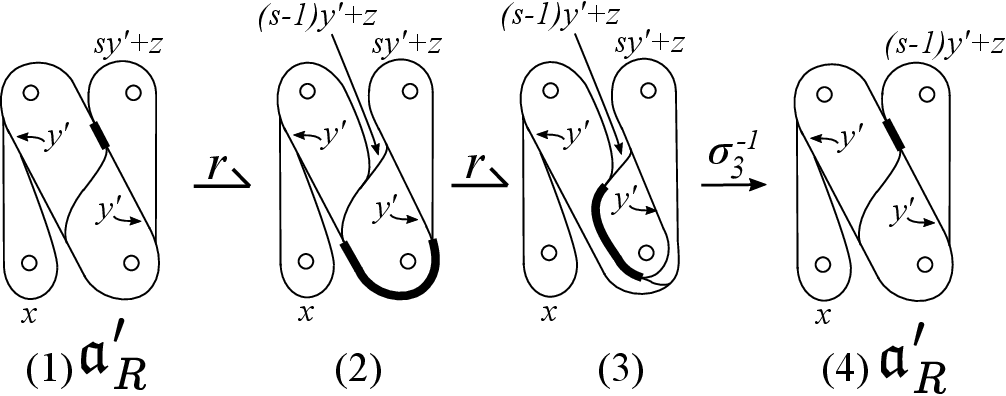}
\caption{Proof of Lemma~\ref{lem_degenerate-3}(a'2). 
(1) $(\mathfrak{a}'_R, M_3^s M_2 \bm{x} = \left(\begin{smallmatrix} x \\ y' \\ sy'+z \end{smallmatrix}\right))$, where $y' = x+y+z$. 
(4) $(\mathfrak{a}'_R, M_3^{s-1} M_2 \bm{x})$.}
\label{fig_mix_s3-part2}
\end{figure}

We  prove claim (a'3). 
Consider the measured train track $(\mathfrak{a}'_R,  M_2 \bm{x})$. 
We may suppose that $x < z$. 
Applying $3$ maximal splittings consecutively, 
we obtain $3$ left maximal splittings 
$ (\mathfrak{a}'_R,  M_2 \bm{x}) \ls  (\mathfrak{s}, M_2 \bm{x}) \ls^2  \sigma_2^{-1} (\mathfrak{b}_L,  \bm{x})$. 
See Figure~\ref{fig_mix_s3-part3}. 
We finished the proof. 
\begin{figure}[ht]
\centering
\includegraphics[height=3.5cm]{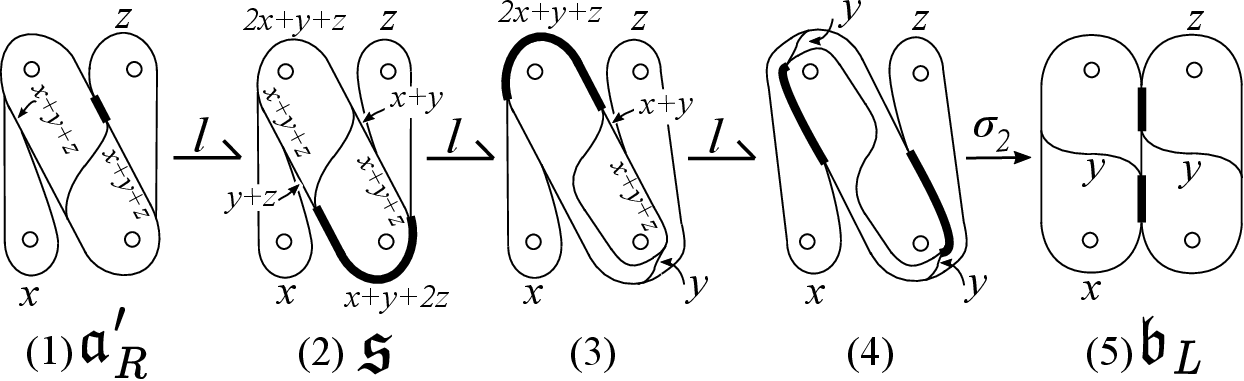}
\caption{Proof of Lemma~\ref{lem_degenerate-3}(a'3). 
(1) $ (\mathfrak{a}'_R,  M_2 \bm{x})$. (2) $(\mathfrak{s}, M_2 \bm{x})$. (4) $(\mathfrak{b}_L,  \bm{x})$.}
\label{fig_mix_s3-part3}
\end{figure}
\end{proof}

Recall that $\Delta = \sigma_1 \sigma_2 \sigma_3 \sigma_1 \sigma_2 \sigma_1$ is the $\pi$-rotation (Figure~\ref{fig_braid_symmetry}(6)).

\begin{lem}
\label{lem_degenerate-1}
Let $q, s \in {\mathbb N}$. 
Let $\bm{x}= \left(\begin{smallmatrix} x \\ y \\ z \end{smallmatrix}\right) >\bm{0}$ 
and $J= \left(\begin{smallmatrix}0 & 0 & 1 \\0 & 1 & 0 \\1 & 0 & 0\end{smallmatrix}\right)$. 
\begin{enumerate}
\item[(a1)] 
$ \Delta (\mathfrak{a}'_R, M_3^{s-1} M_2^q J\bm{x}) = (\sigma_1^{-1} \circ \rs \circ  \lrs) (\mathfrak{b}_L, M_1^s M_2^q \bm{x})$. 

\item[(a2)] 
$\Delta(\mathfrak{a}'_R, M_3^{s-1} M_2^q J \bm{x}) = (\sigma_1^{-1} \circ \rs^2 \circ \Delta) (\mathfrak{a}'_R, M_3^s M_2^q J\bm{x}) $. 

\item[(a3)] 
$(\mathfrak{b}_L,  M_2^{q-1} \bm{x}) = ( \sigma_2 \circ \ls^3 \circ \Delta) (\mathfrak{a}'_R,  M_2^q J \bm{x})$ if $x \ne z$. 
\end{enumerate}
\end{lem}

\begin{proof}
Observe that 
$\Delta (\mathfrak{b}_L, M_3^s M_2^q J \bm{x}) = (\mathfrak{b}_L, M_1^s M_2^q \bm{x})$.  
By  
$\Delta \sigma_i^{\pm 1}= \sigma_j^{\pm 1} \Delta$ for the pair $(i,j)= (1,3)$ or $(3,1)$, 
the proof is analogous to that of Lemma~\ref{lem_degenerate-3}. 
\end{proof}

Let $ (\mathfrak{b}_L, M_1^s M_2^q \bm{x})$ or $(\mathfrak{b}_L, M_3^s M_2^q \bm{x})$ be a measured train track, 
where the measures are preceded by a type A and A' block respectively. We now compute the maximal splittings of the measured train tracks.

\begin{prop}[Type A/A' block for (1)/(2)]
\label{prop_typeC}
Let $q, s \in {\mathbb N}$. 
Let $\bm{x}= \left(\begin{smallmatrix} x \\ y \\ z \end{smallmatrix}\right) >\bm{0}$. 
\begin{enumerate}
\item[(1)] 
$(\mathfrak{b}_L, \bm{x}) = (\sigma_2^q \sigma_1^{-s}  \circ  \ls^{3q} \circ \rs^{2s-1} \circ  \lrs) (\mathfrak{b}_L, M_1^s M_2^q \bm{x})$. 

\item[(2)] 
$(\mathfrak{b}_L, \bm{x}) = (\sigma_2^q \sigma_3^{-s} \circ \ls^{3q} \circ \rs^{2s-1} \circ  \lrs ) (\mathfrak{b}_L, M_3^s M_2^q \bm{x})$. 

\end{enumerate}
\end{prop}

\begin{proof}
By a similar argument as in the proof of Proposition~\ref{prop_typeB}, one can prove claims (1) and (2). 
In the case of Proposition~\ref{prop_typeB}, we used  Lemma~\ref{lem_step-braid}.  
For the proof of  (1) (resp. (2)), we use 
Lemma~\ref{lem_degenerate-1} (resp. Lemma~\ref{lem_degenerate-3}) together with Lemma~\ref{lem_step-braid}(6). 
\end{proof}

\begin{proof}[Proof of Theorem~\ref{thm_5-punctured-sphere-precise}]
As in the proof of Theorem~\ref{thm_2-torus}, for a Perron-Frobenius eigenvector $\bm{v}$ of $M_{\bm{p}}$ 
we define positive vectors 
$\bm{x}^{(0)} \colonequals  \bm{v}$ and 
$\bm{x}^{(i)} \colonequals  M_1^{p_i} M_3^{p_i'} M_2^{q_i} \bm{x}^{(i-1)} $ for $i \in \{1, \dots, n\}$.  
Suppose that $\bm{p}$ is asymmetric. 
By Propositions~\ref{prop_typeB}(2) and \ref{prop_typeC}, we have 
$$
(\mathfrak{b}_L, \bm{x}^{(i-1)}) = (\sigma_2^{q_i} \sigma_1^{-p_i} \sigma_3^{-p_i'} \circ  \rightharpoonup^{A_i+ 3q_i} ) (\mathfrak{b}_L, \bm{x}^{(i)} ) 
\hspace{2mm}\mbox{for\ } i \in \{1, \dots, n\}, 
$$
where $A_i= A_i(\bm{p})$ is the positive integer defined in Section~\ref{section_Introduction}. 
This gives us 
$$(\mathfrak{b}_L, \bm{v} = \bm{x}^{(0)}) = 
(\phi_{\bm{p}}^{-1} \circ   \rightharpoonup^{A_1+ 3q_1}   \circ  \cdots \circ \rightharpoonup^{A_n+ 3q_n}) (\mathfrak{b}_L,  \lambda_{\bm{p}} \bm{v} = \bm{x}^{(n)}).$$ 
This means that 
$$(\mathfrak{b}_0, \mu_0)= (\mathfrak{b}_L, \lambda_{\bm{p}} \bm{v}) \rightharpoonup^{A_n+ 3q_n}  \cdots \rightharpoonup^{A_1+ 3q_1} \phi_{\bm{p}}(\mathfrak{b}_L, \bm{v} ) = 
(\mathfrak{b}_{\ell}, \mu_{\ell})$$
is an Agol cycle of $\phi_{\bm{p}}$ with length $\ell$.  
The consecutive $A_i+ 3q_i$ maximal splittings $\rightharpoonup^{A_i+ 3q_i} $ are given by 
Proposition~\ref{prop_typeB}(2) 
when  the $i$-th block of $\phi_{\bm{p}}$ is of type $B$. 
The maximal splittings are given by Propositions~ \ref{prop_typeC} 
when the $i$-th block is of type $A$ or $A'$. 

The proof of the theorem when $\bm{p}$ is symmetric is left to the reader. 
\end{proof}

\begin{ex}\label{ex_simple_agol_cycle-braid}
We present 2 examples for Agol cycles and their total splitting numbers. 
\begin{enumerate}
\item 
For $\bm{p}= (1,2,1) \in \mathcal{I}_1$ asymmetric, 
an Agol cycle of $\phi_{\bm{p}}$ is given by 
$$(\mathfrak{b}_L, \lambda_{\bm{p}} \bm{v}_{\bm{p}}) \rs\  \ls\   \rs \  \ls^2\  \rs\  \ls^2 \phi_{\bm{p}} (\mathfrak{b}_L, \bm{v}_{\bm{p}})$$ 
whose length is $ 8$. 
The splitting number of each maximal splitting  is $1$, 
except for the first maximal splitting $\rs$ whose splitting number is $2$ 
(Figure~\ref{fig_rl13}(1)(2)).  
Hence, we have $N(\phi_{\bm{p}})= 9$. 

\item
For $\bm{p}= (1,0,1, 0,1,1) \in \mathcal{I}_2$ asymmetric, 
an Agol cycle of $\phi_{\bm{p}}$ is given by 
$$(\mathfrak{b}_L, \lambda_{\bm{p}} \bm{v}_{\bm{p}} ) 
\lrs\ \rs\ \ls^3 \  \lrs\ \rs\ \ls^3 \phi_{\bm{p}}(\mathfrak{b}_L,  \bm{v}_{\bm{p}}),$$
whose length is $10$. 
The splitting number of each maximal splitting  is $1$, 
except for the $2$ mixed maximal splittings $\lrs$, 
whose splitting number is $2$ 
(Figure~\ref{fig_mix_s3}(1)(2)).  
Hence, we have $N(\phi_{\bm{p}}) = 12$. 
\end{enumerate}
\end{ex}

\begin{thm}
\label{thm_total_FD} 
For $ \bm{p}  \in  \mathcal{I}_n$ 
the total splitting number of an Agol cycle of $\phi_{\bm{p}}$ is given by 
We have $N(\phi_{\bm{p}}) = \sum_{i=1}^n (A_i(\bm{p})+ 4q_i)  $. 
\end{thm}

\begin{proof}
For each finite sequence of maximal splittings given by Propositions~\ref{prop_typeB} and \ref{prop_typeC}, 
we compute its total splitting number. 
For instance, take a finite sequence 
$$(\mathfrak{b}_L, M_1^{p_i} M_2^{q_i} {\bm x}) \lrs\  \rs^{2p_i-1}\   \ls^{3q_i} \sigma_1^{p_i} \sigma_2^{-q_i} (\mathfrak{b}_L, \bm{x})$$ 
given by Proposition~\ref{prop_typeC}(1). 
Counting the large branches with maximal weight in each maximal splitting, one sees that 
its total splitting number is $2p_i+ 4q_i (= A_i(\bm{p})+ 4q_i)$. 
One can prove the total splitting number  of the Agol cycle for $\phi_{\bm{p}}$ given by Theorem~\ref{thm_5-punctured-sphere-precise} 
equals  the sum of $A_i(\bm{p})+ 4q_i$ over $i$, that is 
$ \sum_{i=1}^n (A_i(\bm{p})+ 4q_i)$. 
\end{proof}

\begin{proof}[Proof of Theorem~\ref{thm_additive}]
Theorems~\ref{thm_total_FT} and \ref{thm_total_FD} immediately give the desired statement. 
\end{proof}

 \section{Conjugacy classes of pseudo-Anosov maps in  \texorpdfstring{$F_T$}{Lg} and \texorpdfstring{$F_D$}{Lg}}
 \label{section_applications}
 
In the final section 
we 
classify conjugacy classes of pseudo-Anosov maps in the semigroups $F_T$ and $F_D$. 
To do this, we define maps $T: \mathbb{N}_0^{3n} \rightarrow {\mathbb N}_0^{3n}$, called  the {\em shift}, and 
$f: {\mathbb N}_0^{3n}  \rightarrow {\mathbb N}_0^{3n} $, called the {\em flip}, as follows. 
For $ \bm{p}=  (p_n, p_n', q_n, \dots, p_1, p_1', q_1)  \in  {\mathbb N}_0^{3n}$ 
\begin{eqnarray*}
T(\bm{p})&=& (p_{n-1}, p_{n-1}', q_{n-1}, \dots, p_1, p_1', q_1, p_n, p_n', q_n), 
\\
f(\bm{p})&=&  (p_n', p_n, q_n, \dots, p_1', p_1, q_1). 
\end{eqnarray*}
The shift $T$ permutes by three entries and the flip $f$ interchanges $p_i$ and $p_i'$ for all $i \in \{1, \dots, n\}$. 
Note that  $\bm{p}$ is symmetric if and only if 
the flip $f$ preserves $\bm{p}$, i.e.  $f(\bm{p})= \bm{p}$.
 Let $\bm{p}  \in  \mathcal{I}_n$ and $\bm{t}  \in  \mathcal{I}_m$. 
 We write $\bm{p} \sim \bm{t}$ if $n= m$ and $ T^k(\bm{p}) \in \{\bm{t}, f(\bm{t})\}$  for some $k \ge 0$.

\begin{thm}
\label{thm_conjugacy-class}
Let $ \bm{p}  \in  \mathcal{I}_n$ and $\bm{t}  \in  \mathcal{I}_m$. 
The following are equivalent. 
\begin{enumerate}
\item[(1)] 
$\bm{p} \sim \bm{t}$. 

\item[(2)] 
$\Phi_{\bm{p}}$ and $\Phi_{\bm{t}}$ are conjugate in $\mathrm{MCG}(\Sigma_{1,2})$. 

\item[(3)] 
$\phi_{\bm{p}}$ and $\phi_{\bm{t}}$ are conjugate in $\mathrm{MCG}(\Sigma_{0,5})$. 
\end{enumerate}
\end{thm}

\begin{proof}
Suppose that 
$\bm{p} \sim \bm{t}$. 
This means that $T^k(\bm{p})= \bm{t}$ or $T^k(\bm{p})= f(\bm{t})$ for some $k \ge 0$. 
By the definition of the shift $T$, $\Phi_{\bm{p}}$ and $\Phi_{T(\bm{p})}$ (resp. $\phi_{\bm{p}}$ and $\phi_{T(\bm{p})}$) are conjugate.
Note that $\Phi_{\bm{p}}$ and $\Phi_{f(\bm{p})}$ (resp. $\phi_{\bm{p}}$ and $\phi_{f(\bm{p})}$) 
    are also conjugate. 
    In this case, a conjugacy is given by $F$ (resp. $\Delta$),
    where $F: \Sigma_{1,2} \rightarrow \Sigma_{1,2}$ is the $\pi$-rotation along the simple closed curve $c_2$ (Figure~\ref{fig_t_track}(1)).

Thus the condition (1) implies the conditions (2) and (3).

To see the that (2) implies (1), 
suppose that  $\Phi_{\bm{p}}$ and $\Phi_{\bm{t}}$ are conjugate in $\mathrm{MCG}(\Sigma_{1,2})$. 
By Theorem~\ref{thm_Hodgson-Issa-Segerman} their periodic splitting sequences are combinatorially isomorphic  and 
their Agol cycle lengths are equal. 
Notice that by Theorem~\ref{thm_2-torus} 
$\bm{p}$ is symmetric if and only if $\bm{t}$ is symmetric. 
We now prove that $\bm{p} \sim \bm{t}$ when both $\bm{p}$ and $\bm{t}$ are asymmetric. 
(The proof for the symmetric case is analogous.) 
Let $\ell$ be the Agol cycle lengths of $\Phi_{\bm{p}}$ and $\Phi_{\bm{t}}$. 
For $ \bm{p}=  (p_n, p_n', q_n, \dots, p_1, p_1', q_1) \in \mathcal{I}_n$ and 
$\bm{t}=  (t_m, t_m', u_m, \dots, t_1, t_1', u_1) \in \mathcal{I}_m$, 
Theorem~\ref{thm_2-torus} tells us that 
\begin{eqnarray*}
 (\mathfrak{b}, \lambda_{\bm{p}} \bm{v}_{\bm{p}}) \rs^{p_n+p_n'} \ls^{3q_n}  \cdots \rs^{p_1+p_1'} \ls^{3q_1} 
 \Phi_{\bm{p}} (\mathfrak{b},  \bm{v}_{\bm{p}}), 
\\
 (\mathfrak{b}, \lambda_{\bm{t}} \bm{v}_{\bm{t}}) \rs^{t_m+t_m'} \ls^{3u_m}  \cdots \rs^{t_1+t_1'} \ls^{3u_1} 
 \Phi_{\bm{t}} (\mathfrak{b},  \bm{v}_{\bm{t}})
\end{eqnarray*}
form Agol cycles of $\Phi_{\bm{p}}$ and $\Phi_{\bm{t}}$ respectively. 
This together with Remark~\ref{rem_commute} implies that 
the cyclically ordered sets 
$\{(p_n+ p_n', 3q_n), \dots, (p_1+ p_1', 3q_1)\}$ and 
$\{(t_m+ t_m', 3u_m), \dots, (t_1+ t_1', 3u_1)\}$ have to be equal. 
In particular, $n= m$. 
Up to the shift $T$, we may assume that 
\begin{quote}
$(*)$ 
$\bm{p}, \bm{t} \in \mathcal{I}_n$ satisfy 
$p_i+ p_i' = t_i+ t_i'$ and $ q_i =  u_i$  for $i= 1, \dots, n$. 
\end{quote}
The following three cases can occur. 
\begin{enumerate}
\item[Case 1.]
$p_i = t_i$ (and $p_i'= t_i'$) for $i= 1,\dots, n$. 

\item[Case 2.] 
$p_i = t'_i$ (and $p_i'= t_i$) for $i= 1,\dots, n$. 

\item[Case 3.] 
Otherwise,.
\end{enumerate}
In case 1 (resp.  case 2) we have  $\bm{p} = \bm{t}$ 
(resp.  $\bm{p} = f(\bm{t})$).  
In both cases it holds $\bm{p} \sim \bm{t}$. 
We will later show that case 3 cannot occur. 
\medskip

\noindent
{\bf Claim 1.} 
Let $(\mathfrak{b}, \bm{x})$ be a measured train track in $\Sigma_{1,2}$ as in Figure~\ref{fig_t_track}(3). 
Let $h: \Sigma_{1,2} \rightarrow \Sigma_{1,2} $ be an orientation-preserving diffeomorphism preserving  the train track $\mathfrak{b}$. 
Then  
$h(\mathfrak{b}, \bm{x}) = (\mathfrak{b}, \bm{x})$ or $h(\mathfrak{b}, \bm{x}) = (\mathfrak{b}, J\bm{x})$, 
where $J$ is the matrix as in Lemma~\ref{lem_degenerate-1}. 
\medskip
\\
Proof of Claim 1. 
Let $\iota: \Sigma_{1,2} \rightarrow \Sigma_{1,2}$ be the hyperelliptic involution, exchanging the two punctures. 
Let $F: \Sigma_{1,2} \rightarrow \Sigma_{1,2}$ be the $\pi$-rotation as above. 
Then  $\iota(\mathfrak{b}, \bm{x}) = (\mathfrak{b}, \bm{x})$, $F(\mathfrak{b}, \bm{x}) = (\mathfrak{b},J \bm{x})$ 
and $F \circ \iota (\mathfrak{b}, \bm{x}) = (\mathfrak{b},J \bm{x})$. 
Consider any  orientation-preserving diffeomorphism $h: \Sigma_{1,2} \rightarrow \Sigma_{1,2} $ preserving the train track $\mathfrak{b}$. 
Since large branches are mapped to large branches under $h$, 
we observe that 
$h$ is either the identity map $1$, $\iota$, $F$ or  $F \circ \iota= \iota \circ F$. 
This completes the proof. 
\medskip

We turn to case 3. 
For $\bm{p} \in \mathcal{I}_n$ 
let $\bm{v}_{\bm{p}} $ be the normalized eigenvector of $M_{\bm{p}}$ given in Theorem~\ref{thm_expansion}. 
If case 3 occurs, 
we have $s_{\bm{p}} + s_{\bm{t}}  \ne 1$ by Corollary~\ref{cor_same-dilatation}(2) 
and $s_{\bm{p}}  \ne  s_{\bm{t}} $ by Corollary~\ref{cor_same-dilatation}(3). 
In particular, 
$\bm{v}_{\bm{p}} \ne J \bm{v}_{\bm{t}}$ and $\bm{v}_{\bm{p}} \ne \bm{v}_{\bm{t}}$. 
But since by Claim 1, the only possible diffeomorphisms are $1$, $\iota$, $F$ or  $F \circ \iota= \iota \circ F$, a diffeomorphism $h: \Sigma_ {1,2} \rightarrow \Sigma_ {1,2}$ with 
$h(\mathfrak{b}, \bm{v}_{\bm{p}}) =(\mathfrak{b}, c \bm{v}_{\bm{t}}) $ for some constant $c>0$ cannot exist. The periodic splitting sequences of $ \Phi_{\bm{p}}$ and $\Phi_{\bm{t}}$ are not combinatorially isomorphic because they do not satisfy the condition (2) in Definition~\ref{definition_ combinatorially-isomorphic}.  
Therefore, 
$ \Phi_{\bm{p}}$ and $\Phi_{\bm{t}}$ are not conjugate to each other by Theorem~\ref{thm_Hodgson-Issa-Segerman}. 
This contradicts the assumption that $ \Phi_{\bm{p}}$ and $\Phi_{\bm{t}}$ are conjugate. 
Thus case 3 does not occur and the condition (2) implies the condition (1).

To see that (3) implies (1),
 suppose that $\phi_{\bm{p}}$ and $\phi_{\bm{t}}$ are conjugate in $\mathrm{MCG}(\Sigma_{0,5})$. 
For the $2$-fold branched cover $\Sigma_{1,2} \rightarrow \Sigma_{0,5}$ 
 their lifts $\Phi_{\bm{p}}$ and $\Phi_{\bm{t}}$ are conjugate in $\mathrm{MCG}(\Sigma_{1,2})$. 
 Then  $\bm{p} \sim \bm{t}$ by the above argument.  
 This completes the proof. 
  \end{proof}

\section*{Acknowledgement}
The first author would like to thank the organizers of the FrontierLab program at Osaka University. The program made it possible for the author to conduct research in Japan. Without it, the paper could not have been written.
The second author's research was  supported by JSPS KAKENHI Grant Numbers JP21K03247, JP22H01125, JP23H0101.

\bibliographystyle{hamsplain}
\bibliography{agol_cycle}
\end{document}